\documentclass{birkjour}

\usepackage[utf8]{inputenc}
\usepackage[T1]{fontenc}
\usepackage[english]{babel}
\usepackage{mathtools,amssymb,amsthm}
\mathtoolsset{showonlyrefs}
\usepackage{enumitem}

\usepackage{geometry}
\usepackage[pdftex]{hyperref}

\newtheorem{theorem}{Theorem}[section]
\newtheorem{lemma}[theorem]{Lemma}
\newtheorem{proposition}[theorem]{Proposition}
\newtheorem{corollary}[theorem]{Corollary}
\theoremstyle{definition}
\newtheorem{remark}[theorem]{Remark}

\numberwithin{equation}{section}

\begin{document}

\title[Fixed energy relativistic Keplerian orbits]{Periodic orbits with prescribed negative energy \\for relativistic Keplerian problems}

\author[A.~Boscaggin]{Alberto Boscaggin}

\address{
Dipartimento di Matematica ``Giuseppe Peano''\\ 
Universit\`{a} degli Studi di Torino\\
Via Carlo Alberto 10, 10123 Torino, Italy}

\email{alberto.boscaggin@unito.it}

\author[G.~Feltrin]{Guglielmo Feltrin}

\address{
Dipartimento di Scienze Matematiche, Informatiche e Fisiche\\ 
Universit\`{a} degli Studi di Udine\\
Via delle Scienze 206, 33100 Udine, Italy}

\email{guglielmo.feltrin@uniud.it}

\author[D.~Papini]{Duccio Papini}

\address{
Dipartimento di Scienze e Metodi dell'Ingegneria\\
Universit\`{a} degli Studi di Modena e Reggio Emilia\\
Via Giovanni Amendola 2, 42122 Reggio Emilia, Italy}

\email{duccio.papini@unimore.it}

\thanks{Work written under the auspices of the Grup\-po Na\-zio\-na\-le per l'Anali\-si Ma\-te\-ma\-ti\-ca, la Pro\-ba\-bi\-li\-t\`{a} e le lo\-ro Appli\-ca\-zio\-ni (GNAMPA) of the Isti\-tu\-to Na\-zio\-na\-le di Al\-ta Ma\-te\-ma\-ti\-ca (INdAM). 
\\
\textbf{Preprint -- July 2026}} 



\date{}

\dedicatory{}

\begin{abstract}
Using a variational approach, we study the existence of periodic solutions with prescribed energy for the relativistic equation
\begin{equation*}
\dfrac{\mathrm{d}}{\mathrm{d}t}\left(\dfrac{m\dot x}{\sqrt{1-|\dot{x}|^{2}/c^{2}}}\right)
= -\alpha \frac{x}{|x|^{3}} + \nabla W(x),
\qquad x\in\mathbb{R}^{N}\setminus\{0\},
\end{equation*}
where $W$ is a lower-order perturbation of the Kepler potential. 
The main difficulty stems from the fact that the Kepler singularity is critical for the associated Maupertuis functional, lying exactly at the boundary between the weak force and strong force regimes. To overcome the resulting lack of compactness, we use a penalization procedure and develop a suitable min--max scheme combined with a blow-up analysis of near-collision critical sequences.
As a consequence, we establish the existence of periodic solutions on prescribed negative energy levels, obtaining non-perturbative results in every dimension $N\geq 2$.
\end{abstract}

\maketitle

\section{Introduction}\label{section-1}

In this paper, we investigate the existence of periodic solutions with prescribed energy for the differential system
\begin{equation}\label{eq-main}
\dfrac{\mathrm{d}}{\mathrm{d}t}\left(\dfrac{m\dot x}{\sqrt{1-|\dot{x}|^{2}/c^{2}}}\right)= \nabla V(x), \qquad x\in\mathbb{R}^{N}\setminus\{0\},
\end{equation}
where $m>0$, $c > 0$, and $V\colon \mathbb{R}^{N}\setminus\{0\} \to \mathbb{R}$, with $N \geq 2$, is a potential of the form
\begin{equation}\label{eq-Vkep}
V(x) = \dfrac{\alpha}{\vert x \vert} + W(x),
\end{equation}
with $\alpha > 0$ and $W$ a lower-order term, for $x \to 0$, with respect to the Kepler potential $\alpha/\vert x \vert$.
Here, and throughout the paper, the energy of a solution $x$ of \eqref{eq-main} is defined as 
\begin{equation}\label{eq-energia}
\mathcal{E}(x) = mc^{2} \left(\dfrac{1}{\sqrt{1-|\dot{x}|^{2}/c^{2}}}-1\right) - V(x),
\end{equation}
according to the Lagrangian structure of \eqref{eq-main}, cf.~\cite[Remark~2.2]{BoFePa-26}.

Equation \eqref{eq-main} arises naturally in relativistic mechanics and electrodynamics; in particular, the unperturbed case $W\equiv0$ gives rise to the so-called relativistic Kepler problem
\begin{equation}\label{eq-kepler}
\dfrac{\mathrm{d}}{\mathrm{d}t}\left(\dfrac{m\dot x}{\sqrt{1-|\dot{x}|^{2}/c^{2}}}\right)
= -\alpha \frac{x}{|x|^{3}}, \qquad x\in\mathbb{R}^{N}\setminus\{0\},
\end{equation}
describing the motion of a relativistic particle in an attractive Coulomb field; see, among others, \cite{AnBa-71,Bo-04,Ji-13,LeMo-PP,MuPa-06} and the references therein. As is well known, this equation displays a richer, though still integrable, dynamics than the classical Kepler problem. In particular, for every energy value $h \in (-mc^2,0)$, equation \eqref{eq-kepler} admits many periodic solutions, including both circular ones (i.e., $|x| \equiv R$ for some $R>0$) and non-circular precessing rosetta-type orbits; see, e.g., \cite{BoDaFe-2021,Bo-04}.

In recent years, the relativistic Kepler problem \eqref{eq-kepler} has attracted increasing attention, as a paradigmatic model for the application of tools of nonlinear analysis and dynamical systems. In particular, using methods from Hamiltonian dynamics and critical point theory, the existence of periodic solutions with prescribed period has been extensively studied for possibly time-dependent equations of the form
\begin{equation*}
\dfrac{\mathrm{d}}{\mathrm{d}t}\left(\dfrac{m\dot x}{\sqrt{1-|\dot{x}|^{2}/c^{2}}}\right)= \nabla_{\!x} V(t,x), \qquad x\in\mathbb{R}^{N}\setminus\{0\},
\end{equation*}
where $V$ exhibits a Keplerian-type behavior. Several results have been obtained, both in perturbative regimes (namely, when $V$ is a small perturbation of the Kepler potential) \cite{BoDaFe-2021,BoFePa-25,BoFePa-26,Ga-19} and in global ones \cite{ArSp-23,Be-24,BoDaPa-24-jde,BoDaPa-24-pisa}.

On the contrary, in the autonomous case $V(t,x) = V(x)$, the search for periodic solutions $x$ with prescribed energy, namely $\mathcal{E}(x) = h$ for some $h \in \mathbb{R}$, has been much less investigated, and the available literature is rather limited.
More precisely, the only results we are aware of are of purely perturbative nature, namely dealing with potentials of the type $V(x) = \alpha/\vert x \vert + \varepsilon W(x)$ with $\varepsilon$ a small parameter. In this setting, the existence of periodic solutions bifurcating from a fixed manifold of periodic solutions of the unperturbed problem ($\varepsilon = 0$) has been proved, in spatial dimension $N = 2$ or $N = 3$,  both when the unperturbed solutions are circular or non-circular \cite{BoDaFe-24ch,BoFePa-25,BoFePa-26}.

The aim of this paper is to prove that fixed energy periodic solutions of \eqref{eq-main} actually exist far beyond the perturbative regime, provided that suitable asymptotic and global conditions on the potential $W$ are assumed. For instance, as a corollary of our main result (see Theorem~\ref{th-allN}) we can prove the following.

\begin{theorem}\label{th-intro}
Let $V\colon \mathbb{R}^{N}\setminus\{0\} \to \mathbb{R}$ be a $\mathcal{C}^{2}$-potential of the type \eqref{eq-Vkep}, where 
\begin{equation*}
|x||W(x)| + |x|^{2}|\nabla W(x)| + |x|^{3} |D^{2}W(x)| \to 0, \quad \text{as $|x|\to 0$.}
\end{equation*}
Furthermore, suppose that, for some $\lambda > 0$,
\begin{equation*}
0 < \lambda W(x) \leq - \langle \nabla W(x), x \rangle + (1-\lambda) \frac{\alpha}{\vert x \vert},  \quad \text{for every $x\in\mathbb{R}^{N}\setminus\{0\}$.}
\end{equation*}
Then, the following holds true:
\begin{itemize}
\item if $N=2$, for every $h\in(-mc^{2},0)$ there exists a periodic solution of \eqref{eq-main} with energy $h$;
\item if $N\geq3$, there exists $h^{*}\in[-mc^{2},0)$ such that for every $h\in(h^{*},0)$ there exists a periodic solution of \eqref{eq-main} with energy $h$.
\end{itemize}
\end{theorem}

Let us stress that, in spite of the fact that a more restrictive condition is needed in the case $N \geq 3$, the constant $h^*$ appearing in the statement can be explicitly computed (cf.~Corollary~\ref{cor-N2-B}). The reason of the difference between the planar case $N = 2$ and the higher dimensional one will be clarified below.

The proof of the above result, and of Theorem~\ref{th-allN} more in general, relies on a global variational method, taking advantage of the recently proved fact \cite{BoDaMH-23} that periodic solutions $x$ of the relativistic system \eqref{eq-main} with fixed energy 
$\mathcal{E}(x) = h$ correspond, after a time reparameterization
\begin{equation*}
x(t) = u(s(t)),
\end{equation*}
to critical points of the Maupertuis-type functional
\begin{equation*}
\mathcal{M}(u) = \int_{0}^{1} |\dot{u}(s)|^{2}\,\mathrm{d}s \, \int_{0}^{1} \biggl(Z_{h}(u(s))+\dfrac{h}{c^{2}}(h+2mc^{2})\biggr)\,\mathrm{d}s,
\end{equation*}
where $Z_h \colon \mathbb{R}^{N}\setminus\{0\} \to \mathbb{R}$ is defined as
\begin{equation}\label{def-Zh-intro}
Z_{h}(u) = 2\left(m+\dfrac{h}{c^{2}}\right) V(u) + \dfrac{1}{c^{2}}V(u)^{2}.
\end{equation}
The relativistic nature of the differential system \eqref{eq-main} is thus reflected in a perturbation, for $c \to +\infty$, of the usual Jacobi metric of classical mechanics, see \cite[Chapter~2, Section~4]{AmCo-93}.

An application of such a relativistic Maupertuis principle was given in \cite{BoDaMH-23} for a potential $V$ of the type
\begin{equation}\label{eq-Vsigma}
V(u) = \dfrac{\alpha}{\vert u \vert^\sigma} + W(u),
\end{equation}
with $\sigma > 1$. In this situation, the leading term at the origin of the potential $Z_h$ is proportional to $\vert u \vert^{-\gamma}$
with $\gamma = 2\sigma > 2$. So, on the one hand, the usual strong force argument (see, for instance, \cite[Lemma~5.3]{AmCo-93}) yields
\begin{equation*}
\int_0^1 Z_h(u_n(s)) \,\mathrm{d}s \to +\infty,  \quad \text{if $u_n \overset{H^1}{\rightharpoonup} u$ with $u^{-1}(0) \neq \emptyset$};
\end{equation*}
and, as a consequence,
\begin{equation}\label{intro-cond1}
\mathcal{M}(u_n)\to +\infty,  \quad \text{if $u_n \overset{H^1}{\rightharpoonup} u$ with $u^{-1}(0) \neq \emptyset$ and $u \not\equiv 0$.}
\end{equation}
On the other hand, as carefully explained in \cite[Remark~3.3]{BoDaMH-23}, it holds that
\begin{equation}\label{intro-cond2a}
\lim_{u \to 0} \left( Z_h(u) + \frac12 \langle \nabla Z_h(u),u \rangle \right) = - \infty
\end{equation}
and so
\begin{align}
\mathcal{M}'(u_n)u_n 
&= 2 \int_{0}^{1} |\dot{u}_n|^{2} \int_{0}^{1} \biggl( Z_{h}(u_n)+  \dfrac{1}{2} \langle \nabla Z_{h}(u_n), u_n\rangle +\dfrac{h}{c^{2}}(h+2mc^{2}) \biggr)
< 0,
\\
& \quad \text{ if $u_n \overset{L^\infty}{\to} 0$ and $\dot u_n \not\equiv 0$.}
\label{intro-cond2}
\end{align}
Conditions \eqref{intro-cond1} and \eqref{intro-cond2} have an essential role in controlling the behavior of the functional $\mathcal{M}$ along a sequence $(u_n)$ approaching a singularity; as a consequence, the existence of periodic solutions with any positive energy $h > 0$ can be proved. In particular, in the planar case $N=2$ such solutions can be characterized as global minima on the class of non-contractible loops in the punctured plane $\mathbb{R}^2 \setminus \{0\}$.

On the contrary, in the case of a Keplerian-type potential (that is, $\sigma = 1$ in \eqref{eq-Vsigma}, cf.~\eqref{eq-Vkep}), the problem becomes much more subtle. Indeed, while property \eqref{intro-cond1} remains true, conditions \eqref{intro-cond2a} and \eqref{intro-cond2}  now fail. In particular, since the order of homogeneity of the leading term of $Z_h$ is balanced by the one of the kinetic term $\Vert \dot u_n \Vert_{L^2}$, the behavior of the functional along a sequence $(u_n)$ approaching a singularity becomes less transparent. We stress that this is not a matter of technicality: indeed, differently from the case $\sigma > 1$ in which periodic solutions have positive energy, here periodic solutions have to be found in the energy range $(-mc^2,0)$, as shown by the analysis of the pure Keplerian case $W\equiv 0$.
It is worth observing, moreover, that these solutions cannot be local minima, not even in dimension $N=2$. Indeed, focusing again on the pure Keplerian case, the Maupertuis functional writes as
\begin{equation*}
\mathcal{M}(u) = \frac{1}{c^2}\int_{0}^{1} |\dot{u}|^{2} \int_{0}^{1} \left( \frac{\alpha^2}{\vert u \vert^2} +  \frac{2\alpha (h+mc^2)}{\vert u \vert} + h (h+2mc^2)\right),
\end{equation*}
and so, for any fixed non-constant critical point $\bar{u}$ of $\mathcal{M}$, the function $m_{\bar{u}}(\lambda) \coloneqq \mathcal{M}(\lambda \bar u)$ is a quadratic polynomial
$m_{\bar{u}}(\lambda) = A \lambda^2 + B \lambda + C$
with 
\begin{equation*}
A = \frac{h(h+2mc^2)}{c^2} \int_{0}^{1} |\dot{\bar{u}}|^{2} < 0.
\end{equation*}
Hence, the parabola parametrized by $m_{\bar{u}}$ is downward-opening, showing that $\bar{u}$ is actually a maximum point 
of $\mathcal{M}$ along the line $\lambda \bar{u}$.

Motivated by the above considerations, in the present paper we develop a general min--max approach to deal with relativistic Keplerian-type problems in arbitrary dimension $N \geq 2$. We now outline our strategy, inspired by the paper \cite{Ta-93}, which analyzes singular weak force problems in classical mechanics as limit of a sequence of problems with strong force.
First, for every $\varepsilon \in (0,1)$ we introduce a penalized Maupertuis functional $\mathcal{M}_\varepsilon \colon \Lambda \to \mathbb{R}$ defined as
\begin{equation*}
\mathcal{M}_\varepsilon(u) = \int_{0}^{1} |\dot{u}|^{2} \int_{0}^{1} \left(Z_{h}(u)+\dfrac{h}{c^{2}}(h+2mc^{2})\right) + \varepsilon \int_{0}^{1} \dfrac{1}{|u|^{2}},
\end{equation*}
where $Z_h$ is as in \eqref{def-Zh-intro} and $\Lambda$ is the open subset of $H^1$-periodic function $u \colon \mathbb{R} \to \mathbb{R}^N \setminus \{0\}$, with the further requirement that, when $N=2$, the winding number of $u$ is equal to $1$.
The role of the above penalization term is to ensure that $\mathcal{M}_\varepsilon(u_n) \to +\infty$ even when the weak limit of $u_n$ is the null function $u \equiv 0$, cf.~\eqref{intro-cond1}. As a consequence, also the Palais--Smale condition holds true.

Second, we look for critical points of $\mathcal{M}_\varepsilon$ at the critical level
\begin{equation*}
b_\varepsilon = \inf_{\gamma \in \Gamma} \max_{y \in D} \mathcal{M}_\varepsilon(\gamma(y)),
\end{equation*}
where $D = [0,1] \times \mathbb{S}^{N-2}$ is a cylinder (just to present all the cases in a unified manner, we adopt in the next lines the unusual convention $\mathbb{S}^0 = \{1\}$, cf.~Remark~\ref{casoN=2})
and $\Gamma$ is a suitable set of continuous functions $\gamma \colon D \to \Lambda$ with a fixed value on the bases 
$B = (\{0\} \times \mathbb{S}^{N-2}) \cup  (\{1\} \times \mathbb{S}^{N-2})$ of the cylinder.
Note that in the case $N=2$ the cylinder degenerates into the interval $[0,1] \times \{1\}$, so that the class $\Gamma$ is just a set of paths
with a fixed value on $B = \{0,1\} \times \{1\}$, as in the classical mountain pass scenario. 
One can prove that, for some constant $b_*$,
\begin{equation}\label{stima-intro}
b_\varepsilon \geq b_* > \frac{4 \pi^2 \alpha^2}{c^2},
\end{equation}
which, together with the behavior of the functional $\mathcal{M}_\varepsilon$ on $\gamma(B)$, yield the existence of a critical point $u_\varepsilon$ of $\mathcal{M}_\varepsilon$ at level $b_\varepsilon$. Moreover, since the cylinder $D$ has dimension $N-1$, the Morse index $\mathrm{i}(u_\varepsilon)$ is less than or equal to $N-1$, as well.

Finally, to conclude the proof we show that the limit as $\varepsilon \to 0^+$ provides a collisionless critical point of the unpenalized Maupertuis functional $\mathcal{M} = \mathcal{M}_0$ as a strong $H^1$-limit of the family $\{u_\varepsilon\}$. This is actually the key step of the proof: arguing by contradiction, it is first shown that $u_\varepsilon \to 0$ strongly. Then, the blow-up family $w_\varepsilon = u_\varepsilon / \Vert \dot u_\varepsilon \Vert_{L^2}$ is introduced and, using $Z_h(u) \sim \alpha^2/(c^2|u|^2)$ for $u \to 0$, it is proved that 
\begin{equation*}
\mathcal{M}_\varepsilon(u_\varepsilon) = \frac{\alpha^2}{c^2} J(w_\varepsilon) + o(1),
\end{equation*}
where $J$ is the scale invariant functional
\begin{equation}\label{defJ-intro}
J(w) = \int_0^1 \vert \dot w \vert^2 \int_0^1 \frac{1}{\vert w \vert^2},
\quad \text{as $\varepsilon \to 0^+$,}
\end{equation}
and that $w_\varepsilon$ converges to a critical point $w_0$ of $J$.
The non-trivial critical levels of $J$ are known explicitly and are given by $4\pi^2 k^2$ for $k \in \mathbb{Z} \setminus \{0\}$. 
At this point, from the fact that the winding number of $u_\varepsilon$ is one if $N=2$, and from the Morse index estimate $\mathrm{i}(u_\varepsilon) \leq N-1$ if $N \geq 3$, one proves that it must be $k = 1$, and so $J(w_0) = 4\pi^2$. Hence, the contradiction 
\begin{equation*}
\frac{4 \pi^2 \alpha^2}{c^2} < b_* \leq \mathcal{M}_\varepsilon(u_\varepsilon) = \frac{\alpha^2}{c^2} J(w_\varepsilon) + o(1) = 
\frac{\alpha^2}{c^2} J(w_0) + o(1),
\quad \text{as $\varepsilon \to 0^+$,}
\end{equation*}
is obtained, finally concluding the proof. 

It is worth mentioning that the proof of estimate \eqref{stima-intro} is different if $N=2$ or $N \geq 3$.
Indeed, in the first case one can bound from below the leading term of $\mathcal{M}_\varepsilon$ using a geometric inequality valid for the functional $J$: precisely, $J(u) \geq 4\pi^2 k^2$ for every non-contractible planar loop which winds $k$ times around the origin (implying that the critical points of $J$ are indeed local minima). In the case $N \geq 3$, a similar inequality is not available (indeed, critical points have non-zero Morse index) and the proof of \eqref{stima-intro} actually requires different arguments. This is the reason why the assumptions in the case $N \geq 3$ are more restrictive: that is, $h > h^*$ in Theorem~\ref{th-intro}. In the more general setting of Theorem~\ref{th-allN}, the proof of estimate \eqref{stima-intro} relies on condition \eqref{estim-N2}, which again differs if $N=2$ or $N \geq 3$.

\smallskip

The plan of the paper is the following. In Section~\ref{section-2}, we present the statement of the main result (Theorem~\ref{th-allN}) together with some corollaries. Section~\ref{section-3} is devoted to the relativistic Maupertuis principle and the functional $J$. The proof of the main theorem is given in Section~\ref{section-4}.

\smallskip

\noindent
\textbf{Notation.} Throughout the paper, $H$ stands for the Sobolev space
of $H^1$ and $1$-periodic functions with values in $\mathbb{R}^N$.
For $u\in H$, we denote by
\begin{equation}
[u] = \int_{0}^{1} u(t)\,\mathrm{d}t
\end{equation}
the integral mean value.

\section{Statement of the main result and corollaries}\label{section-2}

Let us introduce the following set of assumptions concerning the potential  $V\colon \mathbb{R}^{N}\setminus\{0\} \to \mathbb{R}$:
\begin{enumerate}[leftmargin=26pt,labelsep=8pt,label=\textup{$(V_{1})$}]
\item There exist $\alpha>0$ and $W\in\mathcal{C}^{2}(\mathbb{R}^{N}\setminus\{0\})$ such that
\begin{equation*}
V(u) = \dfrac{\alpha}{|u|} + W(u), \quad \text{for every $u\in\mathbb{R}^{N}\setminus\{0\}$,}
\end{equation*} 
with
\begin{align}
&|u|W(u)\to 0, \quad \text{as $|u|\to 0$,}
\\
&|u|^{2}\nabla W(u)\to 0, \quad \text{as $|u|\to 0$,}
\\
&|u|^{3} D^{2}W(u)\to 0, \quad \text{as $|u|\to 0$.}
\end{align}
\label{hp-V1}
\end{enumerate}

\begin{enumerate}[leftmargin=26pt,labelsep=8pt,label=\textup{$(V_{2})$}]
\item There exists $\lambda>0$ such that
\begin{equation*}
0 < \lambda V(u) \leq - \langle \nabla V(u), u \rangle, \quad \text{for every $u\in\mathbb{R}^{N}\setminus\{0\}$.}
\end{equation*} 
\label{hp-V2}
\end{enumerate}
Note that, if \ref{hp-V2} holds, one can assume that $\lambda>0$ is as small as desired. In particular we will take $\lambda\in(0,1)$.

\begin{remark}\label{rem-hp-V2}
Assumption \ref{hp-V2} is reminiscent of the classical Ambrosetti--Rabinowitz condition often employed when using variational methods and, as is well known, it implies suitable inequalities concerning the potential $V$ at zero and at infinity. Indeed, setting $f(s)=s^{\lambda}|u|^{\lambda}V(su)$, by \ref{hp-V2} we have
\begin{equation*}
f'(s) = s^{\lambda-1}|u|^{\lambda} \bigl[\langle \nabla V(su), su \rangle + \lambda V(su) \bigr] \leq 0.
\end{equation*}
Then, $f(s)\leq f(1)$ for every $s\geq1$ and $f(s)\geq f(1)$ for every $s\in(0,1]$. Therefore, taking $s=1/|u|$, we deduce respectively that $V(u/|u|) \leq V(u) |u|^{\lambda}$ if $|u|\in(0,1]$ and $V(u/|u|) \geq V(u) |u|^{\lambda}$ if $|u|\geq1$, that is
\begin{align*}
&V(u) \geq \min_{|\xi|=1} V(\xi) |u|^{-\lambda}, \quad \text{for every $u\in\mathbb{R}^{N}$ with $|u|\in(0,1]$,}
\\
&V(u) \leq \max_{|\xi|=1} V(\xi) |u|^{-\lambda}, \quad \text{for every $u\in\mathbb{R}^{N}$ with $|u|\geq1$.}
\end{align*} 
Note that the inequality for $|u|\in(0,1]$ is compatible with \ref{hp-V1} since, as already mentioned above, one can take $\lambda\in(0,1)$. Moreover, the inequality for $|u|\geq1$ provides
\begin{equation}\label{eq-infinity}
V(u)\to 0, \qquad \text{as $|u|\to +\infty$.}
\end{equation}
A natural class of potentials satisfying assumption \ref{hp-V2}, as well as the other conditions, is given by
\begin{equation*}
V(u) = \dfrac{\alpha}{|u|} + \dfrac{\gamma}{|u|^{\delta}} + \widetilde{W}(u), \qquad u\in\mathbb{R}^{N}\setminus\{0\},
\end{equation*} 
with $\gamma > 0$, $\delta \in (0,1)$, and the map $\widetilde{W}$ with compact support in $\mathbb{R}^{N}\setminus\{0\}$ and small enough.
\hfill$\lhd$
\end{remark}

In order to state our main result, we finally need:
\begin{enumerate}[leftmargin=26pt,labelsep=8pt,label=\textup{$(V_{3})$}]
\item There exist $\beta\in(0,\alpha]$ and $\eta\in(0,+\infty]$ such that 
\begin{equation*}
V(u)\geq \dfrac{\beta}{|u|}, \quad \text{for every $u\in\mathbb{R}^{N}\setminus\{0\}$ with $|u|<\eta$.}
\end{equation*} 
\label{hp-V3}
\end{enumerate}
Let us note that, in view of \ref{hp-V1}, the above condition is actually satisfied for every $\beta < \alpha$, provided $\eta$ is chosen small enough: however, in what follows a precise quantitative condition involving both $\beta$ and $\eta$ will be needed. 
Precisely, for every $h\in(-mc^{2},0)$ we define the quadratic function $p_h \colon \mathbb{R} \to \mathbb{R}$ as
\begin{equation*}
p_{h}(\rho) =
\begin{cases}
\, h(h+2mc^{2}) \rho^{2} + \beta (h+mc^{2}) \rho + 4\pi^{2}\beta^{2},
&\text{if $N=2$,}
\vspace{2pt}\\
\, h(h+2mc^{2}) \rho^{2} + \beta (h+mc^{2}) \rho + \dfrac{\beta^{2}}{4},
&\text{if $N\geq3$,}
\end{cases}
\end{equation*} 
and the positive real number
\begin{equation*}
M_h = \sup_{\rho\in(0,\frac{\eta}{2})} p_{h}(\rho).
\end{equation*}
Note that, since $h\in(-mc^{2},0)$, the graph of $p_{h}$ is a downward-opening parabola, with vertex $(\rho_h,p_h(\rho_h))$ belonging to the first quadrant; moreover, $p_{h}(0) > 0$. As a consequence, $M_h$ is actually equal to $p_h(\rho_h)$ if $\rho_h \leq \eta/2$ and to $p_h(\eta/2)$ otherwise.

\smallskip

Our main result reads as follows.

\begin{theorem}\label{th-allN}
Let $N\geq2$.
Let $V\colon \mathbb{R}^{N}\setminus\{0\} \to \mathbb{R}$ satisfy \ref{hp-V1}, \ref{hp-V2}, and \ref{hp-V3}. Let $h \in(-mc^{2},0)$. If
\begin{equation}\label{estim-N2}
M_h > 4 \pi^{2} \alpha^{2},
\end{equation}
then there exists a periodic solution of \eqref{eq-main} with energy $h$.
\end{theorem}

An immediate corollary of the above result can be obtained when $N = 2$ and condition $\ref{hp-V3}$ is satisfied with $\beta = \alpha$ (that is, $W$ is non-negative near the origin). Indeed in this case $4\pi^2 \alpha^2 =  4\pi^{2}\beta^{2} = p_h(0) < M_h$ and so condition
\eqref{estim-N2} is trivially satisfied for any $h\in(-mc^{2},0)$ (let us note that, on the contrary, for $N \geq 3$ it holds $p_h(0) = \beta^2/4$ and this argument does not work). Precisely, the following result holds true. 
 
\begin{corollary}\label{cor-N2-A}
Let $N=2$. Let $V\colon \mathbb{R}^{N}\setminus\{0\} \to \mathbb{R}$ satisfy \ref{hp-V1} and \ref{hp-V2}. Moreover, assume that 
there exists $\eta\in(0,+\infty]$ such that 
\begin{equation*}
W(u)\geq 0, \quad \text{for every $u\in\mathbb{R}^{N}\setminus\{0\}$ with $|u|<\eta$.}
\end{equation*} 
Then, for every $h\in(-mc^{2},0)$, there exists a periodic solution of \eqref{eq-main} with energy $h$.
\end{corollary}

A further corollary of Theorem~\ref{th-allN}, valid in the general case $N \geq 2$, can be given when in assumption \ref{hp-V3} we can take $\eta = +\infty$. In this situation, as mentioned above, $M_h$ is nothing but the height of the vertex of the parabola described by $p_h$, and we can prove that assumption \eqref{estim-N2} holds true when an explicit bound on $-h$ is imposed. Precisely, rephrasing assumption \ref{hp-V3} (with $\eta = +\infty$) in terms of $W$, the following result holds true.

\begin{corollary}\label{cor-N2-B}
Let $N\geq2$.
Let $V\colon \mathbb{R}^{N}\setminus\{0\} \to \mathbb{R}$ satisfy \ref{hp-V1} and \ref{hp-V2}. Moreover, assume that 
there exists $\beta\in(0,\alpha]$ such that 
\begin{equation*}
W(u)\geq -\dfrac{\alpha-\beta}{|u|}, \quad \text{for every $u\in\mathbb{R}^{N}\setminus\{0\}$.}
\end{equation*} 
Let
\begin{equation*}
h^{*} = - mc^{2} \left(1-\dfrac{\Gamma}{\sqrt{\Gamma^{2}+1}}\right), 
\qquad \text{with \,} \Gamma= 
\begin{cases}
\, \displaystyle 4\pi \sqrt{\dfrac{\alpha^2}{\beta^2}-1},
&\text{if $N=2$,}
\vspace{4pt}\\
\, \displaystyle 4\pi \sqrt{\dfrac{\alpha^2}{\beta^2}-\dfrac{1}{16\pi^2}},
&\text{if $N\geq3$.}
\end{cases}
\end{equation*}
Then, for every $h\in(h^{*},0)$, there exists a periodic solution of \eqref{eq-main} with energy $h$.
\end{corollary}

In particular, taking $\alpha = \beta$ and rephrasing assumption \ref{hp-V2} in terms of $W$, we obtain Theorem~\ref{th-intro}.

\begin{proof}[Proof of Corollary~\ref{cor-N2-B}]
As already observed, by the lower bound on $W$, condition \ref{hp-V3} holds (with $\eta=+\infty$).
To apply Theorem~\ref{th-allN} and conclude, we have only to verify that \eqref{estim-N2} is satisfied for $h\in(h^{*},0)$.

Let $N=2$ (the case $N\geq3$ is analogous and therefore omitted). Recalling that the vertex of the parabola is 
\begin{equation*}
(\rho_h,p_h(\rho_h)) = \left(-\dfrac{\beta (h+mc^{2})}{2h(h+2mc^{2})}, -\dfrac{\beta^{2} (h+mc^{2})^{2}}{4h(h+2mc^{2})} + 4\pi^{2}\beta^{2}\right),
\end{equation*}
we impose that $M_h = p_h(\rho_h)>4\pi^{2} \alpha^{2}$ and obtain
\begin{equation*}
\beta^2 (h + mc^2)^2 + 16 \pi^2 (\alpha^2 - \beta^2) h(h + 2 mc^2) > 0.
\end{equation*}
Expanding with respect to $h$ gives
\begin{equation*}
(\beta^2 + 16 \pi^2 (\alpha^2 - \beta^2)) h^2 
+ 2 mc^2 (\beta^2+ 16 \pi^2 (\alpha^2 - \beta^2)) h 
+ \beta^2 m^2 c^4 > 0.
\end{equation*}
Recalling that $h\in(-mc^{2},0)$, we have
\begin{equation*}
- mc^2 \left(1- \sqrt{\dfrac{16 \pi^2 (\alpha^2 - \beta^2)}{\beta^2 + 16 \pi^2 (\alpha^2 - \beta^2)}} \right) < h < 0,
\end{equation*}
 and the thesis follows.
\end{proof}

\begin{remark}\label{rem-Mercury}
In the case $\alpha=\beta$ and $N=3$ in Corollary~\ref{cor-N2-B}, the ratio $h^*/(mc^{2})$ is approximatively $-3.17 \cdot 10^{-3}$, a number which could appear quite small.
However, for planets of the Solar System the ratio between their relativistic energy and $mc^{2}$ ranges between $-10^{-8}$ (Mercury) and $-10^{-10}$ (Neptune). This suggests that Corollary~\ref{cor-N2-B} can be applied in physically relevant regimes.
\hfill$\lhd$
\end{remark}

\section{Preliminaries}\label{section-3}

In this section, we collect some preliminary materials that will be crucially used in the proof of our main results.
Specifically, in Section~\ref{section-3.1} we provide the statement of the relativistic Maupertuis principle needed for our purposes, while in Section~\ref{section-3.2} we collect some important properties of the auxiliary functional $J$ defined in \eqref{defJ-intro}.

\subsection{A relativistic Maupertuis principle}\label{section-3.1}

Let us consider the relativistic equation
\begin{equation}\label{eq-mainM}
\dfrac{\mathrm{d}}{\mathrm{d}t}\left(\dfrac{m\dot x}{\sqrt{1-|\dot{x}|^{2}/c^{2}}}\right)= \nabla V(x), \qquad x \in \Omega,
\end{equation}
where $\Omega$ is an open set of $\mathbb{R}^{N}$, with $N \geq 2$, and $V \colon \Omega \to \mathbb{R}$ is a function of class $\mathcal{C}^1$. As it is easy to check, solutions of the above equation satisfy an energy conservation principle: precisely, if $x$ is a solution of \eqref{eq-mainM}, then
\begin{equation}\label{eq-energiaM}
\mathcal{E}(x) \coloneqq mc^{2} \left(\dfrac{1}{\sqrt{1-|\dot{x}(t)|^{2}/c^{2}}}-1\right) - V(x(t)) \equiv h,
\end{equation}
for some $h \in \mathbb{R}$. Incidentally, let us note that, if $x$ is a solution of \eqref{eq-mainM} with energy $h$, 
then
\begin{equation}\label{eq-Hillregion}
V(x(t)) + h = mc^{2} \left(\dfrac{1}{\sqrt{1-|\dot{x}(t)|^{2}/c^{2}}}-1\right) \geq 0, \quad \text{ for every $t$.}
\end{equation}
The set $\{x \in \Omega \colon  V(x) + h \geq 0 \}$ is usually known, in classical mechanics, as the Hill's region.

The aim of this section is to prove a suitable version of the Maupertuis principle, that is, proving that (periodic) solutions of equation 
\eqref{eq-mainM} with energy equal to $h$ can be obtained as time-reparameterization of critical points of a suitable functional
on a space of (periodic) functions.
To this end, given $h \in \mathbb{R}$, we define the functional $\mathcal{M}\colon H_\Omega \to\mathbb{R}$ as 
\begin{equation*}
\mathcal{M}(u) = \int_{0}^{1} |\dot{u}(s)|^{2}\,\mathrm{d}s \, \int_{0}^{1} \biggl(Z_{h}(u(s))+\dfrac{h}{c^{2}}(h+2mc^{2})\biggr)\,\mathrm{d}s,
\end{equation*}
where $H_\Omega = \{u \in H \colon u(s) \in \Omega, \text{ for every $s \in [0,1]$}\}$ and
\begin{equation}\label{def-Zh}
Z_{h}(u) = 2\left(m+\dfrac{h}{c^{2}}\right) V(u) + \dfrac{1}{c^{2}}V(u)^{2}.
\end{equation}
For further convenience, we observe that $Z_h$ can also be written as
\begin{equation*}
Z_{h}(u) = \dfrac{1}{c^{2}} \Bigl[V(u)^{2} + 2(h+mc^{2}) V(u) \Bigr],
\end{equation*}
and thus
\begin{equation}\label{eq-Z-h}
Z_{h}(u)+\dfrac{h}{c^{2}}(h+2mc^{2}) = \dfrac{1}{c^{2}} \Bigl[V(u)^{2} + 2(h+mc^{2}) V(u) + h(h+2mc^{2})\Bigr].
\end{equation}
A relativistic Maupertuis principle based on the use of the functional $\mathcal{M}$ was provided in \cite{BoDaMH-23}, assuming that $u$ is a non-constant critical point of $\mathcal{M}$ which satisfies
\begin{equation}\label{hp-Maup-0}
V(u(s)) + h >0, \quad \text{for every $s \in [0,1]$.}
\end{equation}
Such an assumption is trivially satisfied when $V$ is positive and $h$ is non-negative; on the contrary, it may fail in the application considered in this paper: indeed, while $V$ is still positive, $h \in (-mc^2,0)$ and so critical points of $\mathcal{M}$ with $V(u(s)) + h \geq 0$ may exists, cf \eqref{eq-Hillregion}. To address this, we provide a slight generalization of the result in \cite{BoDaMH-23}, together with a more direct proof, assuming instead of \eqref{hp-Maup-0} the weaker assumption \eqref{hp-Maup-pr} below. 

\begin{theorem}\label{th-Maupertuis-principle}
Let $u\in H_\Omega$ be a critical point of $\mathcal{M}$ such that
\begin{equation}\label{hp-Maup-pr}
\mathcal{M}(u) > 0 \quad \text{ and } \quad V(u(s)) + h+2mc^{2}>0, \quad \text{for every $s \in [0,1]$.}
\end{equation}
Then, there exist $T>0$ and an increasing diffeomorphism $\sigma\colon [0,T]\to[0,1]$ such that the function $x(t)=u(\sigma(t))$ 
can be extended by $T$-periodicity to a solution of \eqref{eq-mainM} with energy $h$, that is $\mathcal{E}(x)=h$ (cf.~\eqref{eq-energiaM}).
\end{theorem}

\begin{proof}
Let $u\in H_\Omega$ be a critical point of $\mathcal{M}$.
We first observe that, since $\mathcal{M}(u) > 0$ by \eqref{hp-Maup-pr}, $u$ is non-constant and 
$$
\int_{0}^{1} \biggl(Z_{h}(u(s))+\dfrac{h}{c^{2}}(h+2mc^{2})\biggr)\,\mathrm{d}s > 0.
$$
As a consequence, the constant 
\begin{equation}\label{def-omega}
\omega = \left(\dfrac{1}{2}\int_{0}^{1} |\dot{u}|^{2} \bigg{/} \int_{0}^{1}\left(Z_{h}(u)+\dfrac{h}{c^{2}}(h+2mc^{2})\right)\right)^{\!\frac{1}{2}}
\end{equation}
is well-defined and strictly positive. Then, from
\begin{equation*}
\mathcal{M}'(u) v
= 2 \int_{0}^{1} \langle\dot{u}, \dot{v}\rangle \int_{0}^{1}\left(Z_{h}(u)+\dfrac{h}{c^{2}}(h+2mc^{2})\right) 
+ \int_{0}^{1} |\dot{u}|^{2} \int_{0}^{1}\langle\nabla Z_{h}(u),v\rangle = 0,
\end{equation*}
for every $v\in H$, we obtain that $u$ is a $1$-periodic solution of the differential equation
\begin{equation}\label{eq-u-Maupertuis}
\ddot{u} = \omega^{2} \, \nabla Z_{h}(u).
\end{equation}
Moreover, by energy conservation we know that
\[
\dfrac{1}{2} |\dot{u}(s)|^{2} - \omega^{2} Z_{h}(u(s)) \equiv k, \quad \text{ for some $k \in \mathbb{R}$,}
\]
so that, integrating the last identity on $[0,1]$ and recalling the definition \eqref{def-omega},
\begin{equation}\label{eq-energy-u-Maupertuis}
\dfrac{1}{2} |\dot{u}(s)|^{2} - \omega^{2} Z_{h}(u(s)) = \omega^{2} \dfrac{h}{c^{2}}(h+2mc^{2}), 
\quad \text{for every $s\in[0,1]$.}
\end{equation}
This implies that
\begin{equation*}
Z_{h}(u(s))+\dfrac{h}{c^{2}}(h+2mc^{2}) = \dfrac{1}{2\omega^{2}} |\dot{u}(s)|^{2} \geq 0, 
\quad \text{for every $s\in[0,1]$,}
\end{equation*}
and since, by \eqref{eq-Z-h},
\begin{equation}\label{eq-Zh-product}
Z_{h}(u(s))+\dfrac{h}{c^{2}}(h+2mc^{2}) = \dfrac{1}{c^{2}}(V(u(s))+h)(V(u(s))+h+2mc^{2}) ,
\end{equation}
from the second condition in hypothesis \eqref{hp-Maup-pr} we finally conclude that 
\begin{equation}\label{eq-Hill2}
V(u(s))+h\geq0, \quad \text{for every $s \in [0,1]$.}
\end{equation}
Let us now consider the initial value problem
\begin{equation}\label{eq-dotsigma-Maupertuis}
\dot{\sigma} = \dfrac{c^{2}}{\omega \sqrt{2}} \dfrac{1}{V(u(\sigma))+h+mc^{2}}, \quad \sigma(0)=0.
\end{equation}
By \eqref{eq-Hill2}, $V(u(\sigma))+h+mc^{2}\geq mc^{2}$ for every $\sigma$ and so $0<\dot{\sigma}<(m\omega\sqrt{2})^{-1}$. As a consequence, the solution $\sigma$ of \eqref{eq-dotsigma-Maupertuis} exists for every $t\in[0,+\infty)$ and is strictly increasing. Moreover, $\lim_{t\to+\infty}\sigma(t)=+\infty$, since there are no equilibria.
Therefore, there exists $T>0$ such that $\sigma(T)=1$. 
Let 
\begin{equation*}
x(t)=u(\sigma(t)), \quad \text{for every $t \in [0,T]$.}
\end{equation*}
We aim to verify that $x$ can be extended by $T$-periodicity to a solution of \eqref{eq-mainM} with energy $h$.
Of course, $x(0) = x(T)$; moreover $\dot x(0) = \dot x(T)$, since $\dot \sigma(0) = \dot\sigma(T)$ by \eqref{eq-dotsigma-Maupertuis}. 
Furthermore
\begin{equation}\label{eq-dotx-Maupertuis}
\dot{x} = \dot{u}(\sigma) \dot{\sigma} = \dot{u}(\sigma)\dfrac{c^{2}}{\omega \sqrt{2}} \dfrac{1}{V(u(\sigma))+h+mc^{2}},
\end{equation}
and so from \eqref{eq-energy-u-Maupertuis} and \eqref{eq-Zh-product} we obtain
\begin{equation}\label{eq-denom-Maupertuis}
\begin{aligned}
1-\dfrac{|\dot{x}|^{2}}{c^{2}} 
&= 1- \dfrac{1}{2\omega^{2}}|\dot{u}(\sigma)|^{2} \dfrac{c^{2}}{(V(u(\sigma))+h+mc^{2})^{2}}
= 1- \dfrac{c^{2} Z_{h}(u)+h(h+2mc^{2})}{(V(u(\sigma))+h+mc^{2})^{2}}
\\
&= 1- \dfrac{(V(u(\sigma))+h)(V(u(\sigma))+h+2mc^{2})}{(V(u(\sigma))+h+mc^{2})^{2}} = \dfrac{m^{2}c^{4}}{(V(u(\sigma))+h+mc^{2})^{2}}.
\end{aligned}
\end{equation}
Then,
\begin{equation*}
\mathcal{E}(x) = mc^{2} \left(\dfrac{1}{\sqrt{1-|\dot{x}|^{2}/c^{2}}}-1\right)=V(x) + h,
\end{equation*}
proving \eqref{eq-energiaM}. 
Next, we verify that $x$ solves \eqref{eq-mainM}. Using \eqref{eq-dotx-Maupertuis} and \eqref{eq-denom-Maupertuis}, we have
\begin{equation*}
\dfrac{m\dot{x}}{\sqrt{1-|\dot{x}|^{2}/c^{2}}} = \dfrac{1}{\omega\sqrt{2}}\dot{u}(\sigma)
\end{equation*}
and so, from \eqref{eq-u-Maupertuis} and \eqref{eq-dotsigma-Maupertuis} we deduce
\begin{align*}
\dfrac{\mathrm{d}}{\mathrm{d}t}\left(\dfrac{m\dot{x}}{\sqrt{1-|\dot{x}|^{2}/c^{2}}}\right)
&= \dfrac{1}{\omega\sqrt{2}} \ddot{u}(\sigma) \dot{\sigma}
\\
&= \dfrac{1}{\omega\sqrt{2}} \omega^{2} \nabla Z_{h}(u(\sigma))  \dfrac{c^{2}}{\omega \sqrt{2}} \dfrac{1}{V(u(\sigma))+h+mc^{2}}
=\nabla V(x).
\end{align*}
Summing up, we have proved that, on the interval $[0,T]$, the function $x$ satisfies both \eqref{eq-mainM} and \eqref{eq-energiaM}; moreover, $x(0) = x(T)$ and $\dot x(0) = \dot x(T)$. This implies that $x$ can be extended as a $T$-periodic solution in the usual sense.
\end{proof}

\subsection{The auxiliary functional $J$}\label{section-3.2}

Let us consider the functional
\begin{equation}\label{def-J}
J(u) = \int_{0}^{1} |\dot{u}(t)|^{2} \,\mathrm{d}t \, \int_{0}^{1} \dfrac{1}{|u(t)|^{2}} \,\mathrm{d}t,
\end{equation}
defined on the open set
\begin{equation*}
\mathcal{O}^N = \bigl{\{} u \in H  \colon \text{$u(t) \neq 0$, for every $t$} \bigr{\}}.
\end{equation*}
Note that, for $N = 2$, the open set $\mathcal{O}^2$ can be naturally decomposed into its connected components as
\begin{equation*}
\mathcal{O}^2 = \bigcup_{k \in \mathbb{Z}} O^2_k, \qquad 
\mathcal{O}^2_k = \bigl{\{} u \in \mathcal{O}^2  \colon \mathrm{rot}(u)=k \bigr{\}},
\end{equation*}
where $\mathrm{rot}(u)$ is the clockwise winding number of the planar path
$u\colon [0,1] \to \mathbb{R}^2 \setminus \{0\}$.
A first easy property is then immediately obtained in the planar case.

\begin{proposition}\label{prop-stima-omotopia}
For every $u \in \mathcal{O}^2_k$, it holds that
\begin{equation*}
J(u) \geq 4\pi^2 k^2.
\end{equation*}
\end{proposition}

\begin{proof}
Let us observe at first that, up to a time-inversion, we can assume $k \geq 0$. Moreover, let us write $u(t) = \rho(t) e^{-i\theta(t)}$, with $\rho$ and $\theta$ continuous, and $\rho>0$.
Then $\vert \dot u(t) \vert^2 = \dot \rho(t)^2 + \rho(t)^2  \dot\theta(t)^2$ and so
\begin{align*}
2\pi k & = \int_0^1 \dot\theta(t)  \,\mathrm{d}t \leq \int_0^1 \frac{\rho(t)\vert \dot\theta(t) \vert}{\rho(t)} \,\mathrm{d}t  
\leq \left(\int_0^1 \frac{\mathrm{d}t }{\rho(t)^2}\right)^{\!\frac{1}{2}} \left(\int_0^1 \rho(t)^2 \dot\theta(t)^2 \mathrm{d}t \right)^{\!\frac{1}{2}} \\
& \leq \left(\int_0^1 \frac{\mathrm{d}t }{\vert u(t) \vert^2} \right)^{\!\frac{1}{2}} \left( \int_0^1 \vert \dot u(t)\vert^2 \mathrm{d}t \right)^{\!\frac{1}{2}},
\end{align*}
whence the conclusion.
\end{proof}

A second key property of the functional $J$ is that it is $0$-positively homogeneous for all $N\geq2$, namely
\begin{equation*}
J(\lambda u ) = J(u), \quad \text{for every $u \in H$ and $\lambda > 0$.}
\end{equation*}
On the one hand, this degeneracy of the functional prevents the use of variational methods to find critical points. On the other hand, it has a dynamical counterpart which makes the analysis much easier. Indeed, let us first recall that, according to the usual Maupertuis principle of classical mechanics, (non-constant) critical points of $J$ correspond after a linear time-reparameterizations to zero-energy periodic solutions of the equation
\begin{equation*}
\ddot z = \frac{z}{\vert z \vert^4}.
\end{equation*}
As is well known (see, for instance, \cite[Chapter~1, Section~2.b]{AmCo-93}) solutions of this problem are circular and have constant angular speed. Accordingly, a very explicit description of the set of critical points can be provided: precisely (cf.~\cite[Lemma~2.6 and Lemma~2.7]{Ta-93}), the following result holds true.

\begin{proposition}\label{prop-J-2}
Let $u \in \mathcal{O}^N$ (with $N\geq 2$) be a critical point of $J$ with $J(u) > 0$ (equivalently, $u$ is non-constant). Then, there exist $k \in \mathbb{Z} \setminus \{0\}$, $r > 0$, and $e_1,e_2 \in \mathbb{R}^N$ with $|e_1| = |e_2| = 1$ and $\langle e_1, e_2 \rangle = 0$ such that
\begin{equation*}
u(t) = \zeta_{k,r}(t) \coloneqq r \left( \cos(2\pi k t) e_1 + \sin(2\pi k t) e_2 \right).
\end{equation*}
Moreover, $J(u) = 4 \pi^2 k^2$ and
\begin{itemize}
\item if $N=2$, then $u$ is a local minimum of $J$ (so, its Morse index is equal to zero),
\item if $N \geq 3$, then the Morse index of $u$ is greater than or equal to $(N-2)(2\vert k \vert - 1)$.
\end{itemize}
\end{proposition}

Note that in \cite{Ta-93} the local minimality of the critical points in the planar case is not stated. However, it follows easily from Proposition~\ref{prop-stima-omotopia}, since $\zeta_{k,r} \in \mathcal{O}^2_k$ and $J(\zeta_{k,r}) = 4\pi^2 k^2$.

\section{Proof of Theorem~\ref{th-allN}}\label{section-4}

As already described in the Introduction, the proof of Theorem~\ref{th-allN} is based on a min--max argument for a penalized Maupertuis functional together with a limit procedure.
This section is devoted to this proof and, more precisely, is organized as follows. In Section~\ref{section-4.1}, we first describe the min--max class. In Section~\ref{section-4.2}, we introduce the $\varepsilon$-penalized functional $\mathcal{M}_{\varepsilon}$ and prove a strong force type property, as well as the validity of the Palais--Smale condition. Then, in Section~\ref{section-4.2b}, we set up the min--max argument to establish the existence of a family $\{u_\varepsilon\}$ of critical points, together with suitable information on the level and the Morse index. Finally, in Section~\ref{section-4.3}, we use a blow-up argument to show that, in the limit $\varepsilon \to 0^+$, one obtains critical points of the Maupertuis functional $\mathcal{M} = \mathcal{M}_0$.

\subsection{The min--max class}\label{section-4.1}

Recalling the definition of the sets $\mathcal{O}^N$ given in Section~\ref{section-3.2}, we set
\begin{equation}\label{def-Lambda}
\Lambda
=
\begin{cases}
\, \displaystyle \mathcal{O}^2_1,
&\text{if $N=2$,}
\\
\, \displaystyle \mathcal{O}^N,
&\text{if $N\geq3$.}
\end{cases}
\end{equation} 
Let us now assume $N=2$ and define the function $\sigma \colon [0,1]\to\mathbb{R}^{2}$ as
\begin{equation*}
\sigma(t) = (-\cos(2 \pi t), \sin(2\pi t)).
\end{equation*}
Note that
\begin{equation}\label{norm-sigma2}
|\sigma(t)|^{2} = 1, 
\quad\text{for every $t\in[0,1]$,}
\end{equation}
and, moreover, $\sigma \in \Lambda$; for further convenience, we also observe that
\begin{equation*}
|\dot{\sigma}(t)|^{2} = 4\pi^{2}, 
\quad\text{for every $t\in[0,1]$.}
\end{equation*}
Now, for every $0 < \rho_{0} < \rho_{1}$ we define
\begin{equation}\label{def-Gamma2}
\Gamma_{\rho_{0},\rho_{1}} = \bigl{\{} \gamma\in\mathcal{C}([0,1],\Lambda) \colon
\gamma(0)=\rho_{0}\sigma, \; \gamma(1)=\rho_{1}\sigma
\bigr{\}}.
\end{equation}

The crucial result needed in the sequel is the following elementary intersection-type property. 

\begin{lemma}\label{lem-intersection-2}
Let $N=2$. For every $\gamma\in\Gamma_{\rho_{0},\rho_{1}}$ and for every $\rho\in(\rho_{0},\rho_{1})$ there exists $\tau\in[0,1]$ such that
\begin{equation*}
\max_{t\in[0,1]} |\gamma(\tau)(t) -  [\gamma(\tau)] | = \rho, \qquad \lvert[\gamma(\tau)]\rvert\leq\rho.
\end{equation*}
\end{lemma}

\begin{proof}
Let $\gamma\in\Gamma_{\rho_{0},\rho_{1}}$. By contradiction we assume that for $\theta\in[0,1]$ it holds that
\begin{equation*}
\max_{t\in[0,1]} |\gamma(\theta)(t) -  [\gamma(\theta)] | = \rho, \qquad \lvert[\gamma(\theta)]\rvert>\rho.
\end{equation*}
Then, the support of the path $\gamma(\theta)$ is contained in the closed ball $\overline{B}_{\rho}([\gamma(\theta)])\subseteq\mathbb{R}^{2}\setminus\{0\}$. Hence $\mathrm{rot}(\gamma(\theta))=0$ and thus $\gamma(\theta)\notin \Lambda = \mathcal{O}^2_1$, a contradiction.
\end{proof}

We now deal with the more delicate case $N\geq3$. Here, we define the function $\sigma \colon\mathbb{S}^{N-2}\times[0,1]\to\mathbb{R}^{N}$ as
\begin{equation}\label{def-sigma}
\sigma(q,t) = \bigl((2+\cos(2\pi t))q_{1}+2,  (2+\cos(2\pi t))q_{2}, \ldots, (2+\cos(2\pi t))q_{N-1}, \sin(2\pi t)\bigr).
\end{equation}
Observe that
\begin{equation}\label{norm-sigma}
|\sigma(q,t)|^{2} = 1 + 4(1+q_{1})(\cos(2\pi t)+2) \geq 1, 
\quad\text{for every $q\in\mathbb{S}^{N-2}$ and $t\in[0,1]$,}
\end{equation}
and so $\sigma(q,\cdot) \in \Lambda$ for every $q\in\mathbb{S}^{N-2}$. Moreover, 
\begin{equation}\label{norm-dot-sigma}
|\dot{\sigma}(q,t)|^{2} = 4\pi^{2}, 
\quad\text{for every $q\in\mathbb{S}^{N-2}$ and $t\in[0,1]$.}
\end{equation}
Now, for every $0 < \rho_{0} < \rho_{1}$ we define
\begin{equation}\label{def-GammaN}
\Gamma_{\rho_{0},\rho_{1}} = \bigl{\{} \gamma\in\mathcal{C}([0,1]\times\mathbb{S}^{N-2},\Lambda) \colon
\gamma(0,q)=\rho_{0}\sigma(q,\cdot), \; \gamma(1,q)=\rho_{1}\sigma(q,\cdot)
\bigr{\}}.
\end{equation}
The intersection-type property corresponding to the one in Lemma~\ref{lem-intersection-2} reads as follows.

\begin{lemma}\label{lem-intersection-3}
Let $N\geq3$. For every $\gamma\in\Gamma_{\rho_{0},\rho_{1}}$ and for every $\rho\in(\rho_{0},\rho_{1})$ there exists $(\tau,q)\in[0,1]\times\mathbb{S}^{N-2}$ such that
\begin{equation*}
\max_{t\in[0,1]} |\gamma(\tau,q)(t) -  [\gamma(\tau,q)] | = \rho, \qquad \lvert[\gamma(\tau,q)]\rvert\leq\rho.
\end{equation*}
\end{lemma}

The proof of this result is given \cite[Appendix]{Ta-93}, based on the fact that 
$\vert \sigma(q,t) - [\sigma(q,\cdot)]\vert =1 $ for every $q$ and $t$ and on the
crucial topological information that the map
\begin{equation*}
\mathbb{S}^1 \times\mathbb{S}^{N-2} \cong [0,1] \times\mathbb{S}^{N-2}\ni (q,t) \mapsto \frac{\sigma(q,t)}{\vert \sigma(q,t) \vert} \in \mathbb{S}^{N-1}
\end{equation*}
has degree (in the sense of the topological degree for maps between compact manifolds, see \cite{Hi-76}) equal to $\pm 1$ (depending on the orientation). Notice however that \cite[Proposition~1.4]{Ta-93} proves the weaker conclusion that 
$\lvert[\gamma(\tau,q)]\rvert\leq3\rho$. Actually, it is not difficult to check that the number $3$ can be replaced by any other number strictly greater than $1$. Finally, by compactness, the result stated as Lemma~\ref{lem-intersection-3} above follows. Since the arguments in \cite[Appendix]{Ta-93} are already quite long and should be basically repeated here, for briefness we avoid the details.

\begin{remark}\label{casoN=2}
It can be useful to note that the case $N=2$ can be, in some sense, recovered by the higher-dimensional one. Indeed, one can consider the definition of $\sigma$ given in \eqref{def-sigma} with $N=2$ (so that $\mathbb{S}^{N-2} = \{\pm 1\}$) and $q_1 =-1$ thus obtaining the function $(-\cos(2 \pi t), \sin(2\pi t))$, which is nothing but $\sigma$ in the planar case. In fact, properties \eqref{norm-sigma} for $q_1 = -1$ reduces to \eqref{norm-sigma2}.
Finally, from a topological point of view, the fact that the function $\sigma$ in the planar case has winding number equal to $1$ corresponds to the property that the degree of the map $\sigma/\vert \sigma \vert \colon \mathbb{S}^1 \times\mathbb{S}^{N-2} \to \mathbb{S}^{N-1}$ is equal to $\pm 1$.
\hfill$\lhd$
\end{remark}

\subsection{The functional $\mathcal{M}_{\varepsilon}$: strong force property and Palais--Smale condition}\label{section-4.2}

In what follows, we consider a potential $V \colon \mathbb{R}^{N}\setminus\{0\} \to \mathbb{R}$ of class $\mathcal{C}^2$ 
satisfying assumptions \ref{hp-V1}, \ref{hp-V2}, and \ref{hp-V3}; moreover, let $h\in(-mc^{2},0)$. 
Given $\varepsilon\in(0,+\infty)$, we define a penalized functional $\mathcal{M}_{\varepsilon} \colon \Lambda  \to \mathbb{R}$ by setting
\begin{equation}\label{def-M-eps-Z}
\mathcal{M}_{\varepsilon}(u) 
= \int_{0}^{1} |\dot{u}(t)|^{2}\,\mathrm{d}t \, \int_{0}^{1} \left(Z_{h}(u(t))+\dfrac{h}{c^{2}}(h+2mc^{2})\right)\,\mathrm{d}t + \varepsilon \int_{0}^{1} \dfrac{1}{|u(t)|^{2}}\,\mathrm{d}t,
\end{equation} 
where $\Lambda \subseteq H$ is as in \eqref{def-Lambda} and $Z_h \colon \mathbb{R}^{N}\setminus\{0\} \to \mathbb{R}$ is given by \eqref{def-Zh}.
Recalling \eqref{eq-Z-h}, we thus have
\begin{equation}\label{def-M-eps-V}
\mathcal{M}_{\varepsilon}(u) 
= \dfrac{1}{c^{2}}\int_{0}^{1} |\dot{u}|^{2} \int_{0}^{1} \Bigl[V(u)^{2} + 2(h+mc^{2}) V(u) + h(h+2mc^{2})\Bigr] + \varepsilon \int_{0}^{1} \dfrac{1}{|u|^{2}}.
\end{equation} 

\begin{remark}\label{prop-Z-h}
For future convenience, let us also observe that
\begin{equation}\label{eq-Mprimo}
\begin{aligned}
\mathcal{M}_{\varepsilon}'(u) v 
&= 
2\int_{0}^{1} \langle\dot{u},\dot{v}\rangle \int_{0}^{1} \left(Z_{h}(u)+\dfrac{h}{c^{2}}(h+2mc^{2})\right)
\\
&\quad+ \int_{0}^{1} |\dot{u}|^{2} \int_{0}^{1} \langle \nabla Z_{h}(u), v \rangle
- 2\varepsilon \int_{0}^{1} \dfrac{\langle u, v \rangle}{|u|^{4}},
\end{aligned}
\end{equation}
and so
\begin{equation}\label{eq-Mprimo-u}
\mathcal{M}_{\varepsilon}'(u) u
= 
\int_{0}^{1} |\dot{u}|^{2} \int_{0}^{1} \biggl( 2\left(Z_{h}(u)+\dfrac{h}{c^{2}}(h+2mc^{2})\right)+  \langle \nabla Z_{h}(u), u \rangle \biggr)
- 2\varepsilon \int_{0}^{1} \dfrac{1}{|u|^{2}}.
\end{equation}
Now
\begin{equation}\label{eq-nablaZh}
\nabla Z_{h}(u) = \dfrac{2}{c^{2}} (V(u)+h+mc^{2}) \nabla V(u),
\end{equation}
and so from assumption \ref{hp-V2} we easily get
\begin{equation*}
\langle \nabla Z_{h}(u), u \rangle + \lambda Z_{h}(u) \leq 0, \quad \text{for every $u\in\mathbb{R}^{N}\setminus\{0\}$,}
\end{equation*} 
with the same $\lambda>0$. Inserting this inequality into \eqref{eq-Mprimo-u} we finally obtain
\begin{equation}\label{est-Mprimo}
\mathcal{M}_{\varepsilon}'(u) u\leq \int_{0}^{1} |\dot{u}|^{2} \int_{0}^{1} \biggl( (2-\lambda)Z_{h}(u)+\dfrac{2h}{c^{2}}(h+2mc^{2}) \biggr)- 2\varepsilon \int_{0}^{1} \dfrac{1}{|u|^{2}}.
\end{equation}
This inequality will be used several times throughout the remainder of the paper.
\hfill$\lhd$
\end{remark}

We begin by establishing the following strong force type property for the functional $\mathcal{M}_{\varepsilon}$.

\begin{proposition}\label{prop-4.x}
Let $(u_{n})_{n}\subseteq \Lambda$ and $u\in\partial\Lambda$ be such that $u_{n}$ converges to $u$ weakly in $H^{1}$. Then $\mathcal{M}_{\varepsilon}(u_{n})\to+\infty$. 
\end{proposition}

\begin{proof}
Let $(u_{n})_{n}\subseteq \Lambda$ with $u_{n}\rightharpoonup u \in \partial\Lambda$. 
By the compact embedding of $H$ into $L^{\infty}$, $u_{n}\to u$ in $L^{\infty}$.
We consider two cases: $u\equiv0$ and $u\not\equiv0$.
In the first case, by applying Fatou's lemma, we find
\begin{align*}
+\infty = \varepsilon \int_{0}^{1} \dfrac{1}{|u|^{2}}
&\leq \liminf_{n\to+\infty} \biggl(\varepsilon\int_{0}^{1} \dfrac{1}{|u_{n}|^{2}}\biggr)
\\
&\leq \liminf_{n\to+\infty} \biggl(\mathcal{M}_{\varepsilon}(u_{n}) - \dfrac{h}{c^{2}} (h+2mc^{2}) \int_{0}^{1} |\dot{u}_{n}|^{2} \biggr).
\end{align*} 
Since $\|\dot{u}_{n}\|_{L^{2}}$ is bounded (by the weak convergence), the thesis follows.
In the second case, there exist $t_{0}\in[0,1)$ and $\delta>0$ such that 
\begin{equation*}
u(t_{0})=0, \qquad u(t_{0}+\delta)\neq0, \qquad |u_{n}(t)|\leq\eta, \quad \text{for every $t\in[t_{0},t_{0}+\delta]$,}
\end{equation*}
for $n$ large enough.
Then, by the Cauchy--Schwarz inequality and hypothesis \ref{hp-V3}, we obtain
\begin{align}
&\bigl{|}\log|u_{n}(t_{0}+\delta)|-\log|u_{n}(t_{0})|\bigr{|}
\leq
\int_{t_{0}}^{t_{0}+\delta} \dfrac{|\dot{u}_{n}|}{|u_{n}|}
\\
&\leq \biggl( \int_{t_{0}}^{t_{0}+\delta} |\dot{u}_{n}|^{2} \biggr)^{\!\frac{1}{2}} \biggl( \int_{t_{0}}^{t_{0}+\delta} \dfrac{1}{|u_{n}|^{2}} \biggr)^{\!\frac{1}{2}}
\label{eq--}
\\
&\leq \biggl( \int_{0}^{1} |\dot{u}_{n}|^{2} \biggr)^{\!\frac{1}{2}} \biggl( \int_{0}^{1} \dfrac{V(u_{n})^{2}}{\beta^{2}}\biggr)^{\!\frac{1}{2}}
\\
&\leq \dfrac{1}{\beta} \biggl( c^{2}\mathcal{M}_{\varepsilon}(u_{n})-c^{2} h(h+2mc^{2}) \int_{0}^{1} |\dot{u}_{n}|^{2} \biggr)^{\!\frac{1}{2}},
\end{align}
where in the last inequality we use \eqref{def-M-eps-V} and the fact that $V(u)>0$ for every $u\in\mathbb{R}^{N}\setminus\{0\}$ by \ref{hp-V2}.
Passing to the limit and recalling again that $\|\dot{u}_{n}\|_{L^{2}}$ is bounded, the thesis directly follows also in this second case and the proof is complete.
\end{proof}

In the next proposition, we prove that the functional $\mathcal{M}_{\varepsilon}$ satisfies the Palais--Smale condition at any positive level.

\begin{proposition}\label{prop-Palais-Smale}
Let $b>0$ and $(u_{n})_{n}\subseteq\Lambda$ be such that $\mathcal{M}_{\varepsilon}(u_{n})\to b$ and $\mathcal{M}'_{\varepsilon}(u_{n})\to0$. Then $(u_{n})_n$ possesses a subsequence converging to some $u\in\Lambda$ strongly in $H^{1}$.
\end{proposition}

\begin{proof}
Let $(u_{n})_{n}\subseteq\Lambda$ be such that $\mathcal{M}_{\varepsilon}(u_{n})\to b>0$ and $\mathcal{M}'_{\varepsilon}(u_{n})\to0$.
As a consequence, there exist $0<m<M$ such that $\mathcal{M}_{\varepsilon}(u_{n})\in[m,M]$ for all $n$ large enough and moreover $\|\mathcal{M}'_{\varepsilon}(u_{n})\|_{H^{*}}\to0$.

First we claim that $(u_{n})_{n}$ is bounded in $H^{1}$.
We assume by contradiction that, up to subsequences, $\|u_{n}\|_{H^{1}}\to+\infty$ and one of the following two cases holds: 
\begin{itemize}
\item[$(a)$] $|[u_{n}]| \geq 2\|\dot{u}_{n}\|_{L^{2}}$, for every $n$,
\item[$(b)$] $|[u_{n}]| < 2\|\dot{u}_{n}\|_{L^{2}}$, for every $n$.
\end{itemize}
In the case $(a)$, we deduce that $|[u_{n}]| \to+\infty$, since otherwise $\|\dot{u}_{n}\|_{L^{2}}$ and thus $\|u_{n}\|_{L^{2}}$ would be bounded.
By the mean value theorem for integrals applied to the scalar function $|u_{n}|$, let $t_{n}\in[0,1]$ be such that $|u_{n}(t_{n})|=\int_{0}^{1} |u_{n}(t)|\,\mathrm{d}t$. For every $t\in[0,1]$
\begin{align*}
|u_{n}(t)|
&= \Biggl{|}u_{n}(t_{n})+\int_{t_{n}}^{t} \dot{u}_{n}(s)\,\mathrm{d}s \Biggr{|}
\geq |u_{n}(t_{n})| - \Biggl{|}\int_{t_{n}}^{t} \dot{u}_{n}(s)\,\mathrm{d}s\Biggr{|}
\\
&\geq \int_{0}^{1} |u_{n}(t)|\,\mathrm{d}t - \|\dot{u}_{n}\|_{L^{2}}
\geq |[u_{n}]| - \|\dot{u}_{n}\|_{L^{2}}.
\end{align*}
Therefore, we have
\begin{equation*}
\min_{t\in\mathbb{R}} |u_{n}(t)| \geq |[u_{n}]| - \|\dot{u}_{n}\|_{L^{2}} \geq \dfrac{1}{2} |[u_{n}]| \to +\infty.
\end{equation*} 
Since by \eqref{eq-infinity}
\begin{equation*}
\int_{0}^{1} V(u_{n}) \to 0, 
\quad \int_{0}^{1} \dfrac{1}{|u_{n}|^{2}}\to 0, 
\quad \text{as $n\to+\infty$,}
\end{equation*}
we obtain 
\begin{equation*}
\liminf_{n\to+\infty} \mathcal{M}_{\varepsilon}(u_{n})
\leq h(h+2mc^{2}) \liminf_{n\to+\infty} \|\dot{u}_{n}\|_{L^{2}}^{2}
\leq 0,
\end{equation*}
a contradiction with the fact that $\mathcal{M}_{\varepsilon}(u_{n})\in[m,M]$ for all $n$ large.

In the case $(b)$, since $\|u_{n}\|_{H^{1}}$ is equivalent to the norm $|[u_{n}]| + \|\dot{u}_{n}\|_{L^{2}}$, then there is a constant $\nu>0$ such that
\begin{equation*}
\|u_{n}\|_{H^{1}} \leq \nu ( |[u_{n}]| + \|\dot{u}_{n}\|_{L^{2}}) \leq 3 \nu \|\dot{u}_{n}\|_{L^{2}}.
\end{equation*} 
Since $\delta_{n}\coloneqq\|\mathcal{M}_{\varepsilon}'(u_{n})\|_{H^{*}}\to0$, we have
\begin{equation*}
|\mathcal{M}_{\varepsilon}'(u_{n}) u_{n} | 
\leq \delta_{n}\|u_{n}\|_{H^{1}} 
\leq 3\nu \delta_{n}\|\dot{u}_{n}\|_{L^{2}}.
\end{equation*} 
Recalling \eqref{est-Mprimo} with $\lambda\in(0,1)$ (see Remark~\ref{rem-hp-V2}), 
from
$-3\nu\delta_{n} \|\dot{u}_{n}\|_{L^{2}}\leq \mathcal{M}_{\varepsilon}'(u_{n}) u_{n}$
it follows that
\begin{equation*}
-3\nu\delta_{n} \dfrac{1}{\|\dot{u}_{n}\|_{L^{2}}}  - \dfrac{2h}{c^{2}} (h+2mc^{2})\leq
(2-\lambda) \int_{0}^{1} Z_{h}(u_{n}).
\end{equation*}
Therefore, passing to the inferior limit leads to
\begin{equation*}
-\dfrac{2h}{c^{2}}(h+2mc^{2})
\leq
(2-\lambda)\liminf_{n\to+\infty} \int_{0}^{1} Z_{h}(u_{n}).
\end{equation*}
Thus, there exists $\mu\in((2-\lambda)/2,1)$ such that for $n$ large
\begin{equation*}
(2-\lambda)\int_{0}^{1} Z_{h}(u_{n})
\geq
-\mu \dfrac{2h}{c^{2}}(h+2mc^{2}) >0.
\end{equation*}
Therefore
\begin{align*}
\mathcal{M}_{\varepsilon}(u_{n}) 
&\geq
\int_{0}^{1} |\dot{u}_{n}|^{2} 
\int_{0}^{1}\left(Z_{h}(u_n) + \dfrac{h}{c^{2}}(h+2mc^{2}) \right)
\\
&\geq
\dfrac{1}{c^{2}}\int_{0}^{1} |\dot{u}_{n}|^{2} 
\int_{0}^{1} \Bigl[ -\dfrac{2\mu}{2-\lambda} h(h+2mc^{2}) + h(h+2mc^{2}) \Bigr] 
\\
&=\dfrac{-h(h+2mc^{2})}{c^{2}}\left(\dfrac{2\mu}{2-\lambda}-1 \right) \int_{0}^{1} |\dot{u}_{n}|^{2} .
\end{align*}
Since $\mathcal{M}_{\varepsilon}(u_{n})$ is bounded we deduce that $\int_{0}^{1}|\dot{u}_{n}|^{2}$ is bounded as well, so we reach a contradiction with the fact that $\|\dot{u}_{n}\|_{L^{2}}\geq \frac{1}{3\nu} \|u_{n}\|_{H^{1}}\to+\infty$.

Having proved that $(u_{n})_{n}$ is bounded in $H^{1}$, we have that, up to a subsequence, $u_{n} \rightharpoonup u$ weakly in $H^{1}$ and strongly in $L^{\infty}$, for some $u\in\Lambda$.
Hence, to obtain the strong convergence in $H^{1}$, we just need to show that $\|u_{n}\|_{H^{1}}\to\|u\|_{H^{1}}$.

Since $\|\mathcal{M}_{\varepsilon}'(u_{n})\|_{H^{*}}\to0$ we have $\mathcal{M}_{\varepsilon}'(u_{n})u_{n}\to 0$.
Next, again from \eqref{est-Mprimo} with $\lambda\in(0,1)$,
\begin{align*}
0 &< 2\varepsilon \int_{0}^{1} \dfrac{1}{|u|^{2}} 
=
\lim_{n\to+\infty} \left(\mathcal{M}_{\varepsilon}'(u_{n})u_{n}+2\varepsilon  \int_{0}^{1} \dfrac{1}{|u_{n}|^{2}} \right)
\\
&\leq \liminf_{n\to+\infty}\int_{0}^{1} |\dot{u}_n|^{2} \int_{0}^{1} \biggl( (2-\lambda)Z_{h}(u_n)+\dfrac{2h}{c^{2}}(h+2mc^{2}) \biggr)
\\
&= \liminf_{n\to+\infty}\int_{0}^{1} |\dot{u}_n|^{2}  \lim_{n\to+\infty} \int_{0}^{1} \biggl( (2-\lambda)Z_{h}(u_n)+\dfrac{2h}{c^{2}}(h+2mc^{2}) \biggr),
\end{align*}
and so
\begin{equation*}
-\dfrac{2h}{c^{2}}(h+2mc^{2}) 
\leq (2-\lambda)\lim_{n\to+\infty} \int_{0}^{1} Z_{h}(u_{n})
= (2-\lambda) \int_{0}^{1} Z_{h}(u).
\end{equation*}
Thus, as before,
\begin{equation}\label{Zhu-0}
\int_{0}^{1}\left(Z_{h}(u) + \dfrac{h}{c^{2}}(h+2mc^{2}) \right)
\geq\dfrac{-h(h+2mc^{2})}{c^{2}}\left(\dfrac{2}{2-\lambda}-1 \right)>0.
\end{equation}

From \eqref{eq-Mprimo}, we deduce that
\begin{align*}
\mathcal{M}_{\varepsilon}'(u_{n})(u_{n}-u)
&=
2\int_{0}^{1} \langle \dot{u}_{n}, \dot{u}_{n}-\dot{u}\rangle \int_{0}^{1} \left(Z_{h}(u) + \dfrac{h}{c^{2}}(h+2mc^{2}) \right)
\\
&\quad+ \int_{0}^{1} |\dot{u}_{n}|^{2} \int_{0}^{1} \langle \nabla Z_{h}(u_{n}), u_{n}-u \rangle
-2\varepsilon \int_{0}^{1} \dfrac{\langle u_{n}, u_{n}-u \rangle}{|u_{n}|^{4}}.
\end{align*}
Recalling that $\|\mathcal{M}_{\varepsilon}'(u_{n})\|_{H^{*}}\to0$ and since the second and third term tend to zero, we have
\begin{equation*}
\int_{0}^{1} |\dot{u}_{n}|^{2} \int_{0}^{1} \left(Z_{h}(u_n) + \dfrac{h}{c^{2}}(h+2mc^{2}) \right)
\to
\int_{0}^{1} |\dot{u}|^{2} \int_{0}^{1} \left(Z_{h}(u) + \dfrac{h}{c^{2}}(h+2mc^{2}) \right),
\end{equation*}
which, together with \eqref{Zhu-0}, implies that $\|u_{n}\|_{H^{1}}\to\|u\|_{H^{1}}$. 
Then the proof is complete.
\end{proof}

\subsection{The functional $\mathcal{M}_{\varepsilon}$: geometry and existence of critical points}\label{section-4.2b}

Recalling the definition of the class $\Gamma_{\rho_{0},\rho_{1}}$ given in \eqref{def-Gamma2} for $N=2$ and in \eqref{def-GammaN} for $N \geq 3$, we define the min--max level
\begin{equation*}
b_{\varepsilon,\rho_{0},\rho_{1}}
=
\begin{cases}
\, \displaystyle \inf_{\gamma\in\Gamma_{\rho_{0},\rho_{1}}} \max_{\tau\in[0,1]} \mathcal{M}_{\varepsilon}(\gamma(\tau)),
&\text{if $N=2$,}
\\
\, \displaystyle \inf_{\gamma\in\Gamma_{\rho_{0},\rho_{1}}} \max_{(\tau,q)\in[0,1]\times\mathbb{S}^{N-2}} \mathcal{M}_{\varepsilon}(\gamma(\tau,q)),
&\text{if $N\geq3$.}
\end{cases}
\end{equation*} 
Our first result provides a bound from below for such a min--max level. The validity of this estimate relies crucially on assumption \eqref{estim-N2}.

\begin{proposition}\label{prop-bstar}
There exist $\rho^{*}>0$ and
\begin{equation*}
b_{*}>\dfrac{4\pi^{2}\alpha^{2}}{c^{2}}
\end{equation*}
such that for every $\rho_{0}\in(0,\rho^{*})$, $\rho_{1}>\rho^{*}$, and $\varepsilon\in(0,1)$ it holds that 
\begin{equation}\label{bound_below}
b_{\varepsilon,\rho_{0},\rho_{1}} \geq b_{*}.
\end{equation}
\end{proposition}

\begin{proof}
Let $N=2$. According to assumption \eqref{estim-N2} we fix $\rho^{*}\in(0,\eta/2)$ such that 
\begin{equation*}
p_{h}(\rho^{*}) = h(h+2mc^{2}) (\rho^{*})^{2} + \beta(h+mc^{2}) \rho^{*} + 4\pi^{2}\beta^{2} > 4\pi^{2} \alpha^{2},
\end{equation*}
where $\eta$ is as in \ref{hp-V3}, and $b_{*}=p_{h}(\rho^{*})/c^{2}$.
Let $\rho_{0}\in(0,\rho^{*})$, $\rho_{1}>\rho^{*}$, and $\gamma\in\Gamma_{\rho_0,\rho_1}$. 
By Lemma~\ref{lem-intersection-2} there exists $\hat{\tau}\in[0,1]$ such that
\begin{equation*}
\max_{t\in[0,1]} |\gamma(\hat{\tau})(t) -  [\gamma(\hat{\tau})] | = \rho^{*}, \qquad \lvert [\gamma(\hat{\tau})] \rvert\leq\rho^{*}.
\end{equation*}
Let $u(t) = \gamma(\hat{\tau})(t)$. Then
\begin{equation}\label{eq-est-rho}
\begin{aligned}
&\|\dot{u}\|_{L^{2}} \geq \|\dot{u}\|_{L^{1}} \geq \max_{t\in[0,1]} \lvert u(t) - [u] \rvert = \rho^{*},
\\
&\|u\|_{L^{\infty}} \leq \lvert [u] \rvert + \max_{t\in[0,1]} \lvert u(t) - [u] \rvert \leq 2 \rho^{*}.
\end{aligned}
\end{equation}
Therefore
\begin{align*}
\max_{\tau\in[0,1]} \mathcal{M}_{\varepsilon}(\gamma(\tau))
&\geq \mathcal{M}_{\varepsilon}(u)
\geq \mathcal{M}_{0}(u) 
\\
&= \dfrac{1}{c^{2}} \int_{0}^{1} |\dot{u}|^{2} \int_{0}^{1} \Bigl[V(u)^{2} + 2(h+mc^{2}) V(u) + h(h+2mc^{2})\Bigr]
\\
&\geq \dfrac{1}{c^{2}} \int_{0}^{1} |\dot{u}|^{2} 
\int_{0}^{1} \left( \dfrac{\beta^{2}}{|u|^{2}}
+ 2(h+mc^{2}) \dfrac{\beta}{|u|}
+ h(h+2mc^{2})\right),
\end{align*}
where the last inequality follows from \ref{hp-V3}. Recalling the definition of $J$ in \eqref{def-J} and using \eqref{eq-est-rho}, we thus find
\begin{align*}
\mathcal{M}_{\varepsilon}(u) 
&\geq \dfrac{\beta^{2}}{c^{2}} J(u) + \dfrac{(\rho^{*})^2}{c^{2}} 2(h+mc^{2})  \dfrac{\beta}{2\rho^{*}} + \dfrac{(\rho^{*})^2}{c^{2}} h(h+2mc^{2})
\\
&\geq \dfrac{4\pi^{2}\beta^{2}}{c^{2}} + \dfrac{\beta(h+mc^{2})}{c^{2}} \rho^{*}+ \dfrac{h(h+2mc^{2})}{c^{2}} (\rho^{*})^2 = b_{*},
\end{align*}
where in the last inequality we use Proposition~\ref{prop-stima-omotopia}. Passing to the inferior limit with respect to $\gamma\in\Gamma_{\rho_{0},\rho_{1}}$ we thus obtain $b_{\varepsilon,\rho_{0},\rho_{1}} \geq b_{*}$, which is the thesis.

In the case $N\geq3$, according to assumption \eqref{estim-N2}, as above, we fix $\rho^{*}\in(0,\eta/2)$ such that $p_{h}(\rho^{*}) > 4\pi^{2} \alpha^{2}$ and $b_{*}=p_{h}(\rho^{*})/c^{2}$.
Using Lemma~\ref{lem-intersection-3} instead of Lemma~\ref{lem-intersection-2}, the proof proceeds in a similar way, with the final crucial difference that the term $J(u)$ is estimated using \eqref{eq-est-rho} as $J(u) \geq 1/4$ leading to $\mathcal{M}_{\varepsilon}(u) \geq p_{h}(\rho^{*})/c^{2} = b_{*}$.
\end{proof}

Having proved the bound from below \eqref{bound_below} for the min--max level 
$b_{\varepsilon,\rho_{0},\rho_{1}}$, the next result provides (for $\rho_0$ small enough and $\rho_1$ large enough) a 
geometry of generalized mountain pass type for the functional $\mathcal{M}_{\varepsilon}$.

\begin{proposition}\label{prop-geometria}
Let $\rho^{*}$ be as in Proposition~\ref{prop-bstar}.
There exist $\rho_{0}^{*}$ and $\rho_{1}^{*}$, with $0< \rho_{0}^{*} < \rho^{*} < \rho_{1}^{*}$, and $\varepsilon^{*}>0$ such that for every $\varepsilon\in(0,\varepsilon^{*})$ and for every $\gamma \in \Gamma_{\rho_{0}^*,\rho_{1}^*}$ it holds that
\begin{equation*}
\begin{cases}
\, \max\bigl\{\mathcal{M}_{\varepsilon}(\gamma(0)), \mathcal{M}_{\varepsilon}(\gamma(1))\bigr\} < b_{*},
&\text{if $N=2$,}\vspace{2pt}
\\
\, \displaystyle \max\biggl\{\max_{q\in\mathbb{S}^{N-2}}\mathcal{M}_{\varepsilon}(\gamma(0,q)), \max_{q\in\mathbb{S}^{N-2}}\mathcal{M}_{\varepsilon}(\gamma(1,q))\biggr\} < b_{*},
&\text{if $N\geq3$.}
\end{cases}
\end{equation*}
\end{proposition}

\begin{proof}
In more concrete terms, we aim to prove that there are $\rho_{0}^{*}$ and $\rho_{1}^{*}$, with $0< \rho_{0}^{*} < \rho^{*} < \rho_{1}^{*}$, and $\varepsilon^{*}>0$ such that for every $\varepsilon\in(0,\varepsilon^{*})$
\begin{equation*}
\begin{cases}
\, \max\bigl\{\mathcal{M}_{\varepsilon}(\rho_{0}^{*}\sigma), \mathcal{M}_{\varepsilon}(\rho_{1}^{*}\sigma)\bigr\} < b_{*},
&\text{if $N=2$,}\vspace{2pt}
\\
\, \displaystyle \max\biggl\{\max_{q\in\mathbb{S}^{N-2}}\mathcal{M}_{\varepsilon}(\rho_{0}^{*}\sigma(q,\cdot)), \max_{q\in\mathbb{S}^{N-2}}\mathcal{M}_{\varepsilon}(\rho_{1}^{*}\sigma(q,\cdot))\biggr\} < b_{*},
&\text{if $N\geq3$.}
\end{cases}
\end{equation*}
Let us focus on the more delicate case $N\geq 3$. For $q\in\mathbb{S}^{N-2}$, let us consider
\begin{equation*}
\mathcal{M}_{0}(\rho \sigma(q,\cdot)) 
= \dfrac{\rho^{2}}{c^{2}} \int_{0}^{1} |\dot{\sigma}(q,\cdot)|^{2} \int_{0}^{1} \Bigl[V(\rho \sigma(q,\cdot))^{2} + 2(h+mc^{2}) V(\rho \sigma(q,\cdot)) + h(h+2mc^{2})\Bigr].
\end{equation*} 
Recalling \eqref{norm-sigma}, it holds that 
\begin{equation*}
\lim_{\rho\to+\infty}\min_{t\in[0,1]} \rho |\sigma(q,t)| = +\infty, 
\quad \text{uniformly in $q\in\mathbb{S}^{N-2}$,}
\end{equation*} 
and so, by \eqref{eq-infinity}, there exists $\rho_{1}^{*}>\rho^{*}$ such that
\begin{equation}\label{eq1-prop-geom}
\max_{q\in\mathbb{S}^{N-2}} \mathcal{M}_{0}(\rho_{1}^{*} \sigma(q,\cdot)) < 0.
\end{equation} 
On the other hand, since
\begin{equation*}
\lim_{\rho\to 0^{+}} \rho|\sigma(q,t)| = 0, 
\quad \text{uniformly in $q\in\mathbb{S}^{N-2}$,}
\end{equation*} 
from \ref{hp-V1} we obtain that
\begin{equation*}
\lim_{\rho\to 0^{+}} \mathcal{M}_{0}(\rho \sigma(q,\cdot)) = \dfrac{\alpha^{2}}{c^{2}} \int_{0}^{1} |\dot{\sigma}(q,\cdot)|^{2} \int_{0}^{1} \dfrac{1}{|\sigma(q,\cdot)|^{2}}, 
\quad \text{uniformly in $q\in\mathbb{S}^{N-2}$.}
\end{equation*} 
Taking into account \eqref{norm-sigma} and \eqref{norm-dot-sigma}, we have
\begin{equation*}
\dfrac{\alpha^{2}}{c^{2}} \int_{0}^{1} |\dot{\sigma}(q,\cdot)|^{2} \int_{0}^{1} \dfrac{1}{|\sigma(q,\cdot)|^{2}} \leq \dfrac{4\pi^{2} \alpha^{2}}{c^{2}} < b_{*}, \quad \text{for every $q\in\mathbb{S}^{N-2}$,}
\end{equation*} 
and so there exists $\rho_{0}^{*}\in(0,\rho^{*})$ such that
\begin{equation}\label{eq2-prop-geom}
\max_{q\in\mathbb{S}^{N-2}} \mathcal{M}_{0}(\rho_{0}^{*} \sigma(q,\cdot)) < b_{*}.
\end{equation} 
From \eqref{eq1-prop-geom} and \eqref{eq2-prop-geom}, we can choose $\varepsilon^{*}>0$ such that for every $\varepsilon\in(0,\varepsilon^{*})$ 
\begin{equation*}
\max\biggl\{\max_{q\in\mathbb{S}^{N-2}}\mathcal{M}_{\varepsilon}(\rho_{0}^{*}\sigma(q,\cdot)), \max_{q\in\mathbb{S}^{N-2}}\mathcal{M}_{\varepsilon}(\rho_{1}^{*}\sigma(q,\cdot))\biggr\}  < b_{*}.
\end{equation*}
The case $N=2$ is analogous and actually easier since $\sigma$ does not depend on $q$.
\end{proof}

We are now in a position to conclude. In view of the geometric properties established by Proposition~\ref{prop-bstar} and Proposition~\ref{prop-geometria}, of the strong force type property given by Proposition~\ref{prop-4.x}, and of the fact that the Palais--Smale condition holds at any positive level by Proposition~\ref{prop-Palais-Smale}, standard arguments in critical point theory (cf.~\cite[Proposition~1.5]{Ta-93}) provide, for every $\varepsilon \in (0,\varepsilon^*)$, a critical point $u_\varepsilon$ of the functional $\mathcal{M}_\varepsilon$ satisfying
\begin{equation*}
\mathcal{M}_\varepsilon(u_\varepsilon) = b_{\varepsilon,\rho_{0}^*,\rho_{1}^*}.
\end{equation*}
Note that, by Proposition~\ref{prop-bstar} and setting $b^{*}=b_{\varepsilon^{*},\rho_{0}^*,\rho_{1}^*}$, we have that
\begin{equation*}
b_* \leq \mathcal{M}_{\varepsilon}(u_{\varepsilon}) \leq b^{*}, \quad \text{for every $\varepsilon \in (0,\varepsilon^*)$.}
\end{equation*}
Moreover, arguing again as in the proof of \cite[Proposition~1.5]{Ta-93}, one can assume that, for every $N \geq 2$,
\begin{equation}\label{info-morse}
\mathrm{i}(u_\varepsilon) \leq N-1,
\end{equation}
where by $\mathrm{i}(\cdot)$ we denote the Morse index. This information will actually play a crucial role in the case $N \geq 3$.

\subsection{Limit as $\varepsilon\to0^+$}\label{section-4.3}

The goal of this section is to prove that, for $\varepsilon \to 0^+$, the family $\{u_\varepsilon\}$ converges (possibly passing to a subsequence) to a critical point $u \in \Lambda$ of the functional $\mathcal{M}_0$, with $\mathcal{M}_{0}(u)\geq b_{*} > 0$.
Since it holds that $V(u) + h + 2mc^2 > 0$ by \ref{hp-V2} and $h \in (-mc^2,0)$, 
the Maupertuis principle stated in Theorem~\ref{th-Maupertuis-principle} implies that $u_0$ can be reparameterized as a periodic solution of \eqref{eq-main} with energy $h$, thus concluding the proof of Theorem~\ref{th-allN}.

As a first step, we show the following.

\begin{proposition}\label{prop-alternatives}
There exists $u\in\overline{\Lambda}$ such that, up to a subsequence, $u_{\varepsilon}$ converges to $u$ strongly in $H^{1}$.
Moreover, one and only one of the following alternatives holds: 
\begin{itemize}
\item[$(i)$] $u\equiv0$,
\item[$(ii)$] $u\in\Lambda$ and $u$ is a critical point of $\mathcal{M}_{0}$ with $\mathcal{M}_{0}(u)\geq b_{*}>0$.
\end{itemize}
\end{proposition}

\begin{proof}
We first prove that $\|u_{\varepsilon}\|_{H^{1}}$ is bounded, which implies that, up to subsequences, $u_{\varepsilon}$ converges to $u$ weakly in $H^{1}$, for some $u\in\overline{\Lambda}$.
As a first step, we show that $\|\dot{u}_{\varepsilon}\|_{L^{2}}$ is bounded. Using formula \eqref{est-Mprimo} (with $u_\varepsilon$ instead of $u$ and $\lambda\in(0,1)$) and $\mathcal{M}_{\varepsilon}'(u_{\varepsilon})u_{\varepsilon} =0$, we obtain
\begin{align}
0 < 2\varepsilon \int_{0}^{1} \dfrac{1}{|u_{\varepsilon}|^{2}} 
\leq
\int_{0}^{1} |\dot{u}_{\varepsilon}|^{2} \int_{0}^{1} \biggl( (2-\lambda)Z_{h}(u_{\varepsilon})+\dfrac{2h}{c^{2}}(h+2mc^{2}) \biggr).
\end{align}
Then
\begin{equation*}
(2-\lambda)\int_{0}^{1} Z_{h}(u_{\varepsilon}) \geq - \dfrac{2h}{c^{2}}(h+2mc^{2}),
\end{equation*}
implying that 
\begin{equation}\label{stima-sotto}
\int_{0}^1 \biggl(Z_{h}(u_{\varepsilon})+\dfrac{h}{c^{2}}(h+2mc^{2}) \biggr) \geq \dfrac{-h(h+2mc^{2})}{c^2} \left(\dfrac{2}{2-\lambda}-1 \right)  \eqqcolon \nu_* > 0.
\end{equation}
Therefore
\begin{equation*}
b^{*} \geq \mathcal{M}_{\varepsilon}(u_{\varepsilon}) \geq \int_{0}^{1} |\dot{u}_{\varepsilon}|^{2} \int_{0}^{1} \biggl(Z_{h}(u_{\varepsilon})+\dfrac{h}{c^{2}}(h+2mc^{2}) \biggr) \geq
\nu_* \int_{0}^{1} |\dot{u}_{\varepsilon}|^{2},
\end{equation*}
finally proving that $\|\dot{u}_{\varepsilon}\|_{L^{2}}$ is bounded. As a second step, we show that $\|u_{\varepsilon}\|_{H^{1}}$ is bounded as well. Recalling that the usual $H^1$-norm is equivalent to $|[u_{\varepsilon}]| + \|\dot{u}_{\varepsilon}\|_{L^{2}}$, 
we assume by contradiction that, up to a subsequence, $|[u_{\varepsilon}]|\to+\infty$. Then, since $\|\dot{u}_{\varepsilon}\|_{L^{2}}$ is bounded, we have $\min_{t} |u_{\varepsilon}(t)| \to +\infty$ and thus, by \eqref{eq-infinity},
\begin{equation*}
\limsup_{\varepsilon\to0^{+}} \mathcal{M}_{\varepsilon}(u_{\varepsilon}) \leq 0,
\end{equation*}
which is a contradiction.

In the rest of the proof we show that the convergence of $u_{\varepsilon}$ to $u$ is actually strong and that 
either $(i)$ or $(ii)$ holds true. 

Let us first assume that $u\in\partial\Lambda$. 
Preliminarily, we observe that
\begin{equation}\label{eq-ueps2}
\lim_{\varepsilon\to0^{+}}\int_{0}^{1} \dfrac{1}{|u_{\varepsilon}|^{2}} = +\infty.
\end{equation}
Indeed, if $u\equiv0$ this follows from Fatou's Lemma; otherwise one can argue as in the proof of Proposition~\ref{prop-4.x} (cf.~\eqref{eq--}).
Next, since the weak $H^1$-convergence implies that $\{u_\varepsilon\}$ is bounded in $L^\infty$, from \ref{hp-V1} and \ref{hp-V2} we have, for a suitable constant $C > 0$,
\begin{equation*}
Z_{h}(u_{\varepsilon})+\dfrac{h}{c^{2}}(h+2mc^{2})  \geq \frac{1}{c^2} \Bigl[V(u_\varepsilon)^2 + h(h+2mc^2) \Bigr] \geq \dfrac{\alpha^2}{2c^2|u_\varepsilon|^{2}}-C.
\end{equation*}
Therefore
\begin{equation*}
b^{*} \geq \mathcal{M}_{\varepsilon}(u_{\varepsilon}) \geq \int_{0}^{1} |\dot{u}_{\varepsilon}|^{2} \int_{0}^{1} \left(\dfrac{\alpha^2}{2c^2|u_{\varepsilon}|^{2}}-C \right)
\end{equation*}
and so from \eqref{eq-ueps2} we obtain that $\|\dot{u}_{\varepsilon}\|_{L^{2}}\to 0$, $u$ is constant with $u\equiv0$, and $u_{\varepsilon}\to 0$ strongly in $H^{1}$. In particular, alternative $(i)$ holds.

Finally, assume that $u \in \Lambda$. In this case, from \eqref{eq-Mprimo},
\begin{align*}
0 &= \mathcal{M}_{\varepsilon}'(u_{\varepsilon})(u_{\varepsilon}-u)
=
2\int_{0}^{1} \langle \dot{u}_{\varepsilon}, \dot{u}_{\varepsilon}-\dot{u}\rangle \int_{0}^{1} \biggl(Z_{h}(u_{\varepsilon})+\dfrac{h}{c^{2}}(h+2mc^{2}) \biggr)
\\
&\hspace{95pt}+ \int_{0}^{1} |\dot{u}_{\varepsilon}|^{2} \int_{0}^{1} \langle \nabla Z_{h}(u_{\varepsilon}), u_{\varepsilon}-u \rangle
-2\varepsilon \int_{0}^{1} \dfrac{\langle u_{\varepsilon}, u_{\varepsilon}-u \rangle}{|u_{\varepsilon}|^{4}}.
\end{align*}
Since the second and third term tend to zero by weak convergence, we obtain
\begin{equation*}
\int_{0}^{1} |\dot{u}_{\varepsilon}|^{2} \int_{0}^{1} \biggl(Z_{h}(u_{\varepsilon})+\dfrac{h}{c^{2}}(h+2mc^{2}) \biggr)
\to
\int_{0}^{1} |\dot{u}|^{2} \int_{0}^{1} \biggl(Z_{h}(u)+\dfrac{h}{c^{2}}(h+2mc^{2}) \biggr),
\end{equation*}
which, together with \eqref{stima-sotto}, implies that $\|u_{n}\|_{H^{1}}\to\|u\|_{H^{1}}$ and so the convergence is strong.
From this, the fact that $u$ is a critical point of $\mathcal{M}_0$ with $\mathcal{M}_{0}(u)\geq b_{*}$ follows in a straightforward manner and hence alternative $(ii)$ holds. 
\end{proof}

According to Proposition~\ref{prop-alternatives}, the rest of the section will consist in showing that alternative $(i)$ cannot hold.

Preliminarily to our next arguments, recalling \eqref{def-Zh}, \eqref{eq-nablaZh}, and
\begin{align*}
D^2 Z_h(u)
&=
- \frac{2\alpha^2}{c^2 |u|^4} \mathrm{Id} 
+ \frac{8\alpha^2}{c^2 |u|^6} u \otimes u
- \frac{2\alpha}{c^2 |u|^3} W(u)\mathrm{Id} 
+ \frac{6\alpha}{c^2 |u|^5} W(u) u \otimes u
\\
&\quad
- \frac{4\alpha}{c^2 |u|^3} u \otimes \nabla W(u) 
+ \frac{2\alpha}{c^2 |u|} D^{2} W(u)
+ \frac{2}{c^2} \nabla W(u) \otimes \nabla W(u)
\\
&\quad
+ \frac{2}{c^2}  W(u) D^{2} W(u)
- \frac{2\alpha(h+mc^{2})}{c^2 |u|^{3}} \mathrm{Id} 
+ \frac{6\alpha(h+mc^{2})}{c^2 |u|^{5}} u \otimes u
\\
&\quad
+ \frac{2(h+mc^{2})}{c^2} D^{2} W(u),
\end{align*}
we observe that from assumption \ref{hp-V1} we have the following asymptotic expansions at the origin for the potential $Z_h$ and its first and second derivatives:
\begin{align}\label{eq-1}
Z_h(u) &= \frac{\alpha^2}{c^2 \vert u \vert^2} + o \left( \frac{1}{\vert u \vert^2}\right), \quad \text{as $|u|\to 0$,}
\\
\label{eq-2}
\nabla Z_h(u) &= -\frac{2\alpha^2}{c^2}\frac{u}{\vert u \vert^4} + o \left( \frac{1}{\vert u \vert^3}\right), \quad \text{as $|u|\to 0$,}
\\
\label{eq-3}
D^2 Z_h(u) &= \frac{\alpha^2}{c^2} \left( - \frac{2}{\vert u \vert^4} \mathrm{Id} + \frac{8}{\vert u \vert^6} u \otimes u \right)  + o \left( \frac{1}{\vert u \vert^4}\right), \quad \text{as $|u|\to 0$.}
\end{align}

\begin{proof}[Proof of Theorem~\ref{th-allN}]
We argue by contradiction and assume that alternative $(i)$ of Proposition~\ref{prop-alternatives} holds.
Let us define 
\begin{equation*}
w_{\varepsilon} = \dfrac{u_{\varepsilon}}{\mu_{\varepsilon}},
\qquad \text{with $\mu_{\varepsilon}=\|\dot u_{\varepsilon}\|_{L^{2}}$.}
\end{equation*}
Notice that the definition is well-posed, since $\mu_{\varepsilon}\neq0$. Otherwise, if $\|\dot u_{\varepsilon}\|_{L^{2}}=0$ from \eqref{eq-Mprimo-u} we deduce the contradiction
\begin{equation*}
0 = \mathcal{M}_{\varepsilon}'(u_{\varepsilon}) u_{\varepsilon} = - \dfrac{2\varepsilon}{|u_{\varepsilon}|^{2}}>0.
\end{equation*}
Moreover, $\mu_{\varepsilon}\to0$ as $\varepsilon\to0^{+}$.

\medskip

\noindent
\textbf{Claim~1.} \textit{There exists $w_{0}\in\Lambda$ such that $w_{\varepsilon} \to w_{0}$ weakly in $H^{1}$.}

\smallskip

\noindent
We aim to prove that there exist $c_{1}>0$ and $c_{2}>0$ such that
\begin{equation*}
c_{1} \leq |w_{\varepsilon}(t)| \leq c_{2}, \qquad \text{for every $t\in\mathbb{R}$ and $\varepsilon\in (0,\varepsilon^*)$.}
\end{equation*}
By \eqref{eq-1} we deduce that, for some constant $C_{M}>0$,
\begin{equation*}
Z_{h}(u) \geq \dfrac{\alpha^{2}}{c^{2}|u|^{2}} - C_{M}, \qquad \text{for every $u\in\mathbb{R}^{N}\setminus\{0\}$ with $|u|\leq M$,}
\end{equation*}
and thus
\begin{align*}
b^{*} 
&\geq \mathcal{M}_{\varepsilon}(u_{\varepsilon}) 
\geq \mathcal{M}_{0}(u_{\varepsilon}) 
= \mu_{\varepsilon}^{2} \int_{0}^{1} \left(Z_{h}(\mu_{\varepsilon}w_{\varepsilon}) + \dfrac{h}{c^{2}}(h+2mc^{2}) \right)
\\
&\geq \dfrac{\alpha^{2}}{c^{2}} \int_{0}^{1}\dfrac{1}{|w_{\varepsilon}|^{2}} + \mu_{\varepsilon}^{2} \left(\dfrac{h}{c^{2}}(h+2mc^{2}) -C_{M}\right).
\end{align*}
Therefore,
\begin{equation*}
\left( \int_{0}^{1} \dfrac{1}{|w_{\varepsilon}|^{2}} \right)^{\!\frac{1}{2}} \leq C,
\end{equation*}
for a suitable constant $C>0$. 
A straightforward contradiction argument shows that this inequality implies that there is $C'>0$ such that for each $\varepsilon\in(0,\varepsilon^{*})$ there exists $t_{\varepsilon}\in(0,1)$ such that $|w_{\varepsilon}(t_{\varepsilon})| \geq C'$.
Arguing as in \eqref{eq--}, we have that
\begin{equation*}
\lvert \log|w_{\varepsilon}(t)|-\log|w_{\varepsilon}(t_{\varepsilon})| \rvert \leq \left( \int_{0}^{1} \dfrac{1}{|w_{\varepsilon}|^{2}} \right)^{\!\frac{1}{2}} \leq C,
\end{equation*}
for every $t\in\mathbb{R}$ and $\varepsilon\in(0,\varepsilon^{*})$.
Therefore,
\begin{equation*}
|w_{\varepsilon}(t)| \geq |w_{\varepsilon}(t_{\varepsilon})| e^{-C} \geq C' e^{-C} \eqqcolon c_{1}, \quad \text{for every $t\in[0,1]$.}
\end{equation*}

From \eqref{est-Mprimo} with $\lambda\in(0,1)$, we have
\begin{equation*}
2\varepsilon \int_{0}^{1} \dfrac{1}{|u_{\varepsilon}|^{2}} 
\leq \mu_{\varepsilon}^{2} \int_{0}^{1} \biggl( (2-\lambda)Z_{h}(\mu_{\varepsilon}w_{\varepsilon})+\dfrac{2h}{c^{2}}(h+2mc^{2}) \biggr).
\end{equation*}
Then
\begin{align*}
b_{*} 
&\leq \mathcal{M}_{\varepsilon}(u_{\varepsilon}) 
= \mu_{\varepsilon}^{2} \int_{0}^{1} \left(Z_{h}(\mu_{\varepsilon}w_{\varepsilon}) +\dfrac{2h}{c^{2}}(h+2mc^{2})\right) + \varepsilon \int_{0}^{1} \dfrac{1}{|u_{\varepsilon}|^{2}} 
\\
&\leq \mu_{\varepsilon}^{2} \int_{0}^{1} \left( 2-\dfrac{\lambda}{2} \right) Z_{h}(\mu_{\varepsilon}w_{\varepsilon})+\mu_{\varepsilon}^{2}\dfrac{3h}{c^{2}}(h+2mc^{2}) .
\\
&= \int_{0}^{1} \left( 2-\dfrac{\lambda}{2} \right) \dfrac{\alpha^{2}}{c^{2}|w_{\varepsilon}|^{2}} + o(1),
\end{align*}
where the last equality follows from \eqref{eq-1}.
Since
\begin{equation*}
|w_{\varepsilon}(t)-w_{\varepsilon}(s)| \leq \int_{s}^{t} |\dot{w}_{\varepsilon}| \leq \|\dot{w}_{\varepsilon}\|_{L^{2}} = 1,
\quad \text{for every $s$ and $t$,}
\end{equation*}
then, if $w_{\varepsilon}$ is unbounded, then $\min w_{\varepsilon} \to +\infty$, at least along a subsequence, hence $b_{*}\leq0$, which is a contradiction.

Since $w_{\varepsilon}$ is bounded in $L^{\infty}$, then $w_{\varepsilon}$ is bounded in $L^{2}$ and thus in $H^{1}$. As a consequence, up to a subsequence,
$w_{\varepsilon} \to w_{0}$ weakly in $H^{1}$ for some $w_{0}$ with $c_{1}\leq \|w_{0}\|_{\infty} \leq c_{2}$.

\medskip

\noindent
\textbf{Claim~2.} \textit{It holds that
\begin{equation*}
\lim_{\varepsilon \to 0^+} \dfrac{\varepsilon}{\mu_{\varepsilon}^2} = 0.
\end{equation*}
}

\smallskip

\noindent
From \eqref{eq-Mprimo-u}, \eqref{eq-1}, and \eqref{eq-2}, we deduce that
\begin{align*}
\dfrac{2\varepsilon}{\mu_{\varepsilon}^2 c_{2}^{2}}
\leq
\dfrac{2\varepsilon}{\mu_{\varepsilon}^2}  \int_{0}^{1} \dfrac{1}{|w_{\varepsilon}|^{2}} 
&= 
\mu_{\varepsilon}^2 \int_{0}^{1} \biggl( 2Z_{h}(\mu_{\varepsilon}w_{\varepsilon})+\dfrac{2h}{c^{2}}(h+2mc^{2})+ \langle \nabla Z_{h}(\mu_{\varepsilon}w_{\varepsilon}), \mu_{\varepsilon} w_{\varepsilon} \rangle \biggr)
\\
&= \int_{0}^{1} o\left(\dfrac{1}{|w_{\varepsilon}|^{2}}\right) = o(1),
\end{align*}
for $\varepsilon\to0^{+}$.

\medskip

\noindent
\textbf{Claim~3.} \textit{For $\varepsilon\to0^{+}$, it holds that
\begin{itemize}
\item[$(a)$] $\mathcal{M}_{\varepsilon}(u_{\varepsilon}) - \dfrac{\alpha^{2}}{c^{2}} J(w_{\varepsilon}) \to 0$,
\item[$(b)$] $\mu_{\varepsilon}\mathcal{M}_{\varepsilon}'(u_{\varepsilon}) - \dfrac{\alpha^{2}}{c^{2}} J'(w_{\varepsilon}) \to 0$ (that is, $J'(w_{\varepsilon}) \to 0$),\item[$(c)$] $\mu_{\varepsilon}^{2}\mathcal{M}_{\varepsilon}''(u_{\varepsilon}) - \dfrac{\alpha^{2}}{c^{2}} J''(w_{\varepsilon}) \to 0$.
\end{itemize}
}

\smallskip

\noindent
Concerning point $(a)$, recalling \eqref{def-M-eps-Z} and \eqref{def-J}, we have that
\begin{align*}
&\mathcal{M}_{\varepsilon}(u_{\varepsilon}) - \dfrac{\alpha^{2}}{c^{2}} J(w_{\varepsilon})=
\\
&= \mu_{\varepsilon}^{2} \int_{0}^{1} \left(Z_{h}(\mu_{\varepsilon}w_{\varepsilon}) +\dfrac{2h}{c^{2}}(h+2mc^{2})\right) + \dfrac{\varepsilon}{\mu_{\varepsilon}^{2}} 
 \int_{0}^{1} \dfrac{1}{|w_{\varepsilon}|^{2}} 
- \dfrac{\alpha^{2}}{c^{2}} \int_{0}^{1} \dfrac{1}{|w_{\varepsilon}|^{2}} 
\\
&= \mu_{\varepsilon}^{2} \dfrac{2h}{c^{2}}(h+2mc^{2}) + \dfrac{\varepsilon}{\mu_{\varepsilon}^{2}}  \int_{0}^{1} \dfrac{1}{|w_{\varepsilon}|^{2}} 
+ o(1) = o(1),
\end{align*}
where the last equality follows from \eqref{eq-1} and Claim~2.

For point $(b)$, recalling \eqref{eq-Mprimo}, we observe that
\begin{align*}
&\mu_{\varepsilon}\mathcal{M}_{\varepsilon}'(u_{\varepsilon})v - \dfrac{\alpha^{2}}{c^{2}} J'(w_{\varepsilon})v=
\\
&=
2\mu_{\varepsilon}^{2}\int_{0}^{1} \langle\dot{w}_\varepsilon,\dot{v}\rangle \int_{0}^{1} \left(Z_{h}(\mu_\varepsilon w_\varepsilon)+\dfrac{h}{c^{2}}(h+2mc^{2})\right)
+ \mu_{\varepsilon}^{3} \int_{0}^{1} \langle \nabla Z_{h}(\mu_\varepsilon w_\varepsilon), v \rangle \\
& \quad - 2 \dfrac{\varepsilon}{\mu_{\varepsilon}^{2}}\int_{0}^{1} \dfrac{\langle w_\varepsilon, v \rangle}{|w_\varepsilon|^{4}}
-\dfrac{2\alpha^{2}}{c^{2}}\int_{0}^{1} \langle\dot{w}_\varepsilon,\dot{v}\rangle \int_{0}^{1} \frac{1}{|w_\varepsilon|^2} 
- \dfrac{2\alpha^{2}}{c^{2}}  \int_{0}^{1} \frac{\langle w_\varepsilon, v \rangle}{|w_\varepsilon|^4}
\\
&= 2\dfrac{h}{c^{2}}(h+2mc^{2})\mu_{\varepsilon}^{2}\int_{0}^{1} \langle\dot{w}_\varepsilon,\dot{v}\rangle- 2 \dfrac{\varepsilon}{\mu_{\varepsilon}^{2}}\int_{0}^{1} \dfrac{\langle w_\varepsilon, v \rangle}{|w_\varepsilon|^{4}} +o(1),
\end{align*}
uniformly with respect to $v\in H$ bounded, thanks to \eqref{eq-1} and \eqref{eq-2}.
Analogously, from
\begin{equation*}
\begin{aligned}
\mathcal{M}_{\varepsilon}''(u_\varepsilon)[v,v] 
&= 2\int_{0}^{1} |\dot{v}|^{2} \int_0^1 \left(Z_{h}(u_{\varepsilon}) +\dfrac{2h}{c^{2}}(h+2mc^{2})\right) 
\\
&\quad + 4 \int_{0}^{1} \langle\dot{u}_\varepsilon,\dot{v}\rangle
\int_{0}^{1} \langle \nabla Z_h(u_\varepsilon),v\rangle  +\int_{0}^{1} |\dot{u}_\varepsilon|^{2} \int_{0}^{1} \langle D^2 Z_h(u_\varepsilon)v,v \rangle 
\\
&\quad - 2 \varepsilon \int_{0}^{1} \frac{|v|^2}{|u_\varepsilon|^2} + 8 \varepsilon \int_{0}^{1} \frac{\langle u_\varepsilon,v \rangle^2}{|u_\varepsilon|^6}
\end{aligned}
\end{equation*}
and
\begin{equation*}
\begin{aligned}
J''(w_\varepsilon)[v,v] &= 2\int_{0}^{1} |\dot{v}|^{2} \int_{0}^{1}\frac{1}{|w_\varepsilon|^2} - 4 \int_{0}^{1} \langle\dot{w}_\varepsilon,\dot{v}\rangle \int_{0}^{1} \frac{\langle v_\varepsilon,v\rangle}{|w_\varepsilon|^4} \\
& + 8 \int_{0}^{1}|\dot{w}_\varepsilon|^{2} \int_{0}^{1}\frac{\langle  w_\varepsilon,v\rangle}{|w_\varepsilon|^6} - 2 \int_{0}^{1}|\dot{w}_\varepsilon|^{2}\int_{0}^{1} \frac{|v|^2}{|w_\varepsilon|^4},
\end{aligned}
\end{equation*}
we reach the claim~$(c)$ using \eqref{eq-1}, \eqref{eq-2}, and \eqref{eq-3}.

\medskip

\noindent
\textbf{Claim~4.} \textit{It holds that $w_{\varepsilon} \to w_{0}$ strongly in $H^{1}$.} 

\smallskip

\noindent
By point $(b)$ of Claim~3 and the fact that $w_{\varepsilon}$ converges weakly to $w_0$, we have that
\begin{align*}
0 &= \lim_{\varepsilon\to0^{+}} J'(w_\varepsilon)(w_\varepsilon-w_0)
\\
&= 
\lim_{\varepsilon\to0^{+}} \biggl(
2\int_{0}^{1} \langle\dot{w}_\varepsilon,\dot{w}_\varepsilon-\dot{w}_0\rangle \int_{0}^{1} \frac{1}{|w_\varepsilon|^2} - 2 \int_{0}^{1} |\dot{w}_\varepsilon|^{2} \int_{0}^{1} \frac{\langle w_\varepsilon, w_\varepsilon-w_0 \rangle}{|v_\varepsilon|^4}
\biggr)
\\
&= 
2 \biggl(1-\int_{0}^{1} |\dot{w}_0|^{2} \biggr)\int_{0}^{1} \frac{1}{|w_0|^2}.
\end{align*}
Hence $\|\dot{w}_\varepsilon\|_{L^{2}} \to \|\dot{w}_0\|_{L^{2}} = 1$, from which the strong convergence follows.

\medskip

\noindent
\textbf{Conclusion of the proof.}
By point $(b)$ of Claim~3 and Claim~4, we have that $J'(w_{0})=0$ with $\dot w_0 \not\equiv 0$. Hence by Proposition~\ref{prop-J-2}, $w_{0}=\zeta_{k,r}$ for some $k \neq 0$ and $r>0$.
From $(a)$, $\mathcal{M}_{\varepsilon}(u_{\varepsilon})\to {\alpha^{2}} J(\zeta_{k,r})/{c^{2}} = {4\pi^{2}k^{2} \alpha^{2}}/{c^{2}} $.

In the case $N=2$, since $\mathrm{rot}(u_{\varepsilon})=1$ by definition of $\Lambda$, $\mathrm{rot}(w_{\varepsilon})=\mathrm{rot}(w_{0})=1$. Hence $w_{0}=\zeta_{k,r}$ with $k=1$ and $\mathcal{M}_{\varepsilon}(u_{\varepsilon})\to {4\pi^2\alpha^{2}}/{c^{2}} $. This contradicts the fact that $\mathcal{M}_{\varepsilon}(u_{\varepsilon})\geq b_{*} > {4\pi^2\alpha^{2}}/{c^{2}}$, by Proposition~\ref{prop-bstar}.

In the case $N\geq3$, by point $(c)$ of Claim~3, we have that the Morse index $\mathrm{i}(u_{\varepsilon})$ as a critical point of $\mathcal{M}_\varepsilon$ is greater than or equal to the Morse index of $\zeta_{k,r}$ as a critical point of $J$. Then, by Proposition~\ref{prop-J-2} we have 
\begin{equation*}
\mathrm{i}(u_{\varepsilon})\geq (N-2)(2|k|-1).
\end{equation*}
Combining with \eqref{info-morse}, yields $|k|=1$ and so $\mathcal{M}_{\varepsilon}(u_{\varepsilon})\to {4\pi^2\alpha^{2}}/{c^{2}}$, giving a contradiction as before with Proposition~\ref{prop-bstar}.
The proof is complete.
\end{proof}

\section*{Acknowledgements}
The authors would like to thank Prof.~Walter Dambrosio for helpful discussions on the subject.

%\section*{Declarations}

%\subsubsection*{Conflict of interest} The authors declare that they have no conflict of interest.

%\subsubsection*{Data availability} 
%Data sharing not applicable to this article as no datasets were generated or analyzed during the current study.

\bibliographystyle{elsart-num-sort}
\bibliography{BoFePa-biblio}

\end{document}